\documentclass[10pt, reqno]{amsart}
\usepackage{amsfonts, amssymb,amscd,amsthm,amsxtra}
\usepackage{latexsym}
\usepackage{amsmath}
\usepackage{esint}
\usepackage{graphicx}  \usepackage{epstopdf}

\usepackage{color}
\usepackage[usenames,dvipsnames,svgnames,table]{xcolor}
\bibliography{refs}

\usepackage{hyperref}
\hypersetup{colorlinks = false}

\usepackage[mathscr]{euscript}
\renewcommand{\mathcal}[1]{{\mathscr#1}}

\newtheorem{theorem}{Theorem}[section]

\newtheorem{lemma}[theorem]{Lemma}
\newtheorem{prop}[theorem]{Proposition}

\theoremstyle{definition}
\newtheorem{definition}[theorem]{Definition}
\theoremstyle{remark}
\newtheorem{remark}[theorem]{Remark}
\numberwithin{equation}{section}

\newcommand{\R}{{\mathbb R}}

\renewcommand{\leq}{\leqslant}

\renewcommand{\geq}{\geqslant}

\newcommand{\dist}{\mbox{dist}}
\renewcommand{\epsilon}{\varepsilon }

\newlength{\defbaselineskip}
\newcommand{\setlinespacing}[1]
           {\setlength{\baselineskip}{#1 \defbaselineskip}}

\author[B. Barrios]{Bego\~na Barrios}
\address[Bego\~na Barrios]{Departamento de An\'{a}lisis Matem\'{a}tico, Universidad de La Laguna
\hfill \break \indent C/Astrof\'{\i}sico Francisco S\'{a}nchez s/n, 38271 -- La Laguna, Spain}
\email{bbarrios@ull.es}

\author[M. Medina]{Maria Medina}
\address[Mar\'ia Medina]{Departamento de Matem\'aticas,
Universidad Aut\'onoma de Madrid,
Ciudad Universitaria de Cantoblanco,
28049 Madrid, Spain}
\email{maria.medina@uam.es}

\begin{document}


\keywords{nonlocal operators, fractional $p$-Laplacian, viscosity solutions, weak solutions, non-homogeneous problems.}

\thanks{B. B. was partially supported by the MEC project PGC2018-096422-B-I00 (Spain) and the Ram\'on y Cajal fellowship RYC2018-026098-I (Spain). M. M. was supported by the European Union's Horizon $2020$ research and innovation programme under the Marie Sklodowska-Curie grant agreement N $754446$ and UGR Research and Knowledge Transfer Fund - Athenea3i. M. M. wants to acknowledge Universidad de la Laguna's hospitality, where this work was initiated during a visit in July 2019.}

\title[Equivalence of weak and viscosity solutions]{Equivalence of weak and viscosity solutions in fractional non-homogeneous problems}

\begin{abstract}
We establish the equivalence between the notions of weak and viscosity solutions for non-homogeneous equations whose main operator is the fractional $p$-Laplacian and the lower order term depends on $x$, $u$ and $D_s^p u$, being the last one a type of fractional derivative.

\end{abstract}

\maketitle


\section{Introduction} 
In this work we study the equivalence between the notions of weak and viscosity solutions to the problem 
\begin{equation}\label{problem}
(-\Delta)^s_p u =f(x,u, D_s^pu),
\end{equation}
posed in a bounded domain $\Omega \subset \R^n$, with $s\in (0,1)$ and $1<p<+\infty$. By $(-\Delta)^s_p$ we mean the fractional $p$-Laplacian, defined for a regular enough function $u$ as
\begin{equation}\label{operador}
(-\Delta)^s_p u(x):= a_{n,s,p}\,\mbox{p.v.}\, \int_{\R^n}\frac{|u(x)-u(y)|^{p-2}(u(x)-u(y))}{|x-y|^{n+sp}}dy,
\end{equation}
where $n>ps$ and $a_{n,s,p}$ is a normalization constant that will be omitted from now on. The integral in \eqref{operador} 
has to be understood in the principal value sense, that is, as the limit when $\epsilon\to 0$ of the 
same integral computed in {$\R^n\setminus B_\epsilon(x)$. This operator can be seen as a nonlocal counterpart of the $p$-Laplacian operator,
$$-\Delta_p u:=-\mbox{div} (|\nabla u|^{p-2}\nabla u),$$
and $D_s^p u$ stands for a sort of fractional gradient. Consider indeed the case $p=2$, for which the nonlocal, and linear, operator $(-\Delta)^{s}_{p}$ corresponds to the standard fractional Laplacian given by
$$(-\Delta)^s u(x):=\mbox{p.v.}\int_{\R^n}\frac{u(x)-u(y)}{|x-y|^{n+2s}}dy.$$
As stated in \cite{BPSV, CD}, the corresponding bilinear form
$$\mathcal{B}_s(u,v)(x):=\int_{\R^n}\frac{(u(x)-u(y))(v(x)-v(y))}{|x-y|^{n+2s}}dy,$$
can be seen as a fractional derivative of order $s$ playing the role of $\nabla u\cdot\nabla v$ in the local case. Furthermore, $\mathcal{B}_s(u,u)$ can be thought as the analogue of $|\nabla u|^2$, the natural term appearing in the energy of problems driven by the classical Laplacian. Extending this idea to the whole range $1<p<+\infty$ we define the form, linear only in the second variable, associated to $(-\Delta)_p^su$ as
$$\mathcal{B}^p_s(u,v)(x):=\int_{\R^n}\frac{|u(x)-u(y)|^{p-2}(u(x)-u(y))(v(x)-v(y))}{|x-y|^{n+sp}}dy,$$
and the {\it gradient}
$$D_s^pu(x):=\mathcal{B}_s^p(u,u)(x)={\int_{\R^n}\frac{|u(x)-u(y)|^{p}}{|x-y|^{n+sp}}dy},$$
that correspond to $|\nabla u|^{p-2}\nabla u\cdot \nabla v$ and $|\nabla u|^p$ respectively for the classical $p$-Laplacian. See \cite{AF, BDS, CD} for problems with nonlinearities involving this type of fractional derivative. We refer also the reader to the works \cite{BS, CS, SSS} for other definitions of fractional gradients.

\medskip

Thus, it is clear that \eqref{problem} has a local analogue of the form
\begin{equation}\label{probLocal}
-\Delta_p u(x)=f(x,u,\nabla u)\quad \mbox{ in }\Omega.
\end{equation}
Equivalence among different types of solutions for problem \eqref{probLocal} was firstly studied in \cite{JLM} for the case $f\equiv 0$ via the notion of $p$-harmonic solutions, and later on in \cite{JJ} with different methods (via regularization with infimal convolutions) that apply also to the case $f=f(x)$, a continuous function. Finally in \cite{MO} the problem with a general $f(x,u,\nabla u)$ is studied. By extending the methods of \cite{JJ}, if $f:\Omega\times\R\times \R^n\to\R$ satisfies certain regularity conditions and has a growth of size
\begin{equation}\label{localGrowth}
|f(x,t,\eta)|\leq \gamma(|t|)|\eta|^{p-1}+\phi(x),
\end{equation}
with $\gamma\geq 0$ being a continuous function and $\phi\in L^\infty_{loc}(\Omega)$, the authors prove that viscosity solutions of \eqref{probLocal} are also weak solutions (see \cite[Theorem 1.4]{MO}). The reverse implication, weak solutions being also viscosity, relies on the availability of comparison principles among weak solutions, what imposes certain restrictions on $f$, specially in the singular case $1<p<2$ (see \cite[Theorem 1.5]{MO}). Besides its own interest, this equivalence of the different notions of solutions is useful, for instance, in removability questions (see \cite{JL,Po}).

\medskip

In the nonlocal case, the relation among weak, viscosity and $p$-harmonic solutions has been studied for every $1<p<+\infty$ in the works \cite{KKL, KKP2} when $f\equiv 0$. As far as we know,  concerning the non-homogeneous problem the only result can be found in \cite{SV}, where the authors prove that weak solutions are viscosity when $f=f(x)$ and $p=2$, that is, for the particular case of the fractional Laplacian. The goal of this work is to extend the results in \cite{MO} to the nonlocal framework of \eqref{problem} for every $1<p<+\infty$, giving a much more general equivalence result for nonlocal non-homogeneous problems. The first main result of the paper is the following.}
\begin{theorem}\label{ViscToWeak}
Let $1<p<+\infty$. Assume that $f=f(x,t,\eta)$ is uniformly continuous in $\Omega\times\R\times\R^n$, non increasing in $t$, Lipschitz continuous in $\eta$, and satisfies the growth condition 
\begin{equation}\label{growth}
|f(x,t,\eta)|\leq \gamma (|t|)|\eta|^{\frac{p-1}{p}}+\phi(x),
\end{equation}
where $\gamma\geq 0$ is continuous and $\phi\in L^\infty_{loc}(\Omega)$. Thus, if $u\in L^\infty(\R^n)$ and lower semicontinuous in $\R^n$ is a viscosity supersolution of \eqref{problem} then it is a weak supersolution of the problem.
\end{theorem}

Notice that, since $D_s^pu$ has homogeneity $p$ and $\eta$ corresponds to $D_s^pu$ in problem \eqref{problem}, the growth required in \eqref{growth} is the same as in the local case (see \eqref{localGrowth}).  To prove Theorem \ref{ViscToWeak} we need to pass from the pointwise sense of the viscosity solutions (see Definition \ref{defVisc}) to the integral form of the weak ones (see Definition \ref{defWeak}). We will do it starting, as in \cite{JJ, MO}, by regularizing the function using infimal convolutions. This precise way of regularization will give us useful information on the first and second derivatives, fundamental to carry out convergence arguments. We want to highlight that recovering the weak formulation on the original function is much more involved in the nonlocal setting. Indeed, in the classical framework a key point is to write the operator $\Delta_p$ in the form
$$\Delta_p u=|\nabla u|^{p-2}\left(\Delta u+(p-2)D^2u\frac{\nabla u}{|\nabla u|}\cdot \frac{\nabla u}{|\nabla u|}\right),$$
where $D^2u$ is the Hessian matrix of $u$. This expression permits to use the information on the sign of $D^2u$ in a direct way to apply arguments of convergence of integrals of Fatou's type. However an explicit formulation like this does not hold in the fractional case where, {for an approximating function $u_\delta\to u$}, we are obliged to extract uniform information directly from a term of the form
$$\frac{|u_\delta(x)-u_\delta(y)|^{p-2}(u_\delta(x)-u_\delta(y))}{|x-y|^{n+sp}},$$ 
which is not obvious at all. Furthermore, the fact that $(-\Delta)_p^s$ is an integral operator introduces extra difficulties, since we have to pass to the limit through two integrals, one coming from the weak formulation and one coming form the operator itself. Specially involved is the singular range $1<p\leq 2/(2-s)$. In such case the operator is not well defined even for smooth functions and hence we need to regularize in addition the operator, performing a very delicate argument to pass to the limit (see Lemmas \ref{lemma2} and \ref{lemma22}).
\medskip

The reverse statement, weak solutions being viscosity, is strongly connected with comparison arguments, and the next class of functions needs to be considered.

\begin{definition}Let $u$ be a weak supersolution to \eqref{problem} in $D\subseteq \Omega$. We say that $(u, f)$ satisfies the comparison principle property {\em (CPP)} in $D$ if for every weak subsolution $v$ of \eqref{problem} such that $u\geq v$ a.e. in $\R^n\setminus D$ we have $u\geq v$ a.e. in $D$.
\end{definition}

Although some comparison principles are currently available in the literature for right hand sides of the form $f(x,u)$, this is a delicate issue to be explored in more generality (as far as we know nothing has been proved for nonlinearities involving gradient-type terms). In Subsection \ref{sub_comp} we give a brief description of the known results.

The following implication is proved.

\begin{theorem}\label{WeakToVisc}
Let $1<p<+\infty$. Assume $u$ is a continuous weak supersolution of \eqref{problem} and $f=f(x,t,\eta)$ {is continuous in $\Omega\times\R\times\R^n$} and Lipschitz continuous in $\eta$. If $(CPP)$ holds then $u$ is a viscosity supersolution of \eqref{problem}.
\end{theorem}

The proof of this result runs by contradiction, and relies on the comparison of $u$ with an appropriate perturbation of itself. The nonlocal nature of the fractional derivative $D^p_s$ does not match well with standard arguments of perturbation by a constant. Instead, it forces us to consider and deal with very particular perturbations (constructed by adding precise $C^2$ functions) that are more suitable to the integral operators.

\medskip

We would finally like to mention that problems involving operator \eqref{operador} have concentrated a large interest in the last years. Besides the previously mentioned \cite{KKL, KKP2}, existence and uniqueness of viscosity solutions, as well as convergence to the $p$-Laplace equation, are studied in \cite{IN} and different regularity results for weak solutions can be found in, for instance, \cite{BPV, BCF, BL, dCLP, dCLP2, IMS, KMS, KMS2, P} (see also the references therein). We also want to point out that the results in this work can be straightforward generalized for a larger class of integro-differential operators, as in \cite{KKL, KKP}. They hold for every operator $\mathcal{L}_K$ defined as
$$\mathcal{L}_Ku(x):=\mbox{p.v.}\int_{\R^n}|u(x)-u(y)|^{p-2}(u(x)-u(y))K(x,y)dy,$$
and the corresponding derivative
$$D_{s,K}^pu(x):= \int_{\R^n}|u(x)-u(y)|^pK(x,y)dy,$$
where the kernel $K$ satisfies the following properties:
\begin{itemize}
\item[(i)] $K(x,y)=K(y,x)$ for all $x,y\in\R^n$.
\item[(ii)] $K(x+z,y+z)=K(x,y)$ for all $x,y,z\in \R^n$, $x\neq y$.
\item[(iii)] $\Lambda^{-1}\leq K(x,y)|x-y|^{n+sp}\leq \Lambda$ for all $x,y\in \R^n$, $x\neq y$, where $\Lambda\geq 1$ is a constant.
\item[(iv)] The map $x\mapsto K(x,y)$ is continuous in $\R^n\setminus \{y\}$.
\end{itemize}
The fractional $p$-Laplacian corresponds to the case $K(x,y)=|x-y|^{-n-sp}$ and, for simplicity of the exposition, we will restrict ourselves to this case.

\medskip

The rest of the paper is organized as follows: in Section 2 we introduce the notions of weak and viscosity solutions and also some preliminarie results like a Caccioppoli type inequality and a  continuity property of the operator $D_s^p$. Section 3 deals with the approximation by infimal convolutions and the proof of Theorem \ref{ViscToWeak}. In Section 4 we prove Theorem \ref{WeakToVisc} and we make a brief review on the comparison principles in the literature. 
\medskip

Along the work $C$ will denote a positive constant that may change from line to line.

\section{Preliminaries}
\subsection{Notions of solutions}
Given a measurable function $u:\mathbb{R}^{n}\to\mathbb{R}$ we define the Gagliardo seminorm associated to a, possibly unbounded, domain $\Omega$ as
$$[u]_{s,p,\Omega}:=\left(\iint_{\Omega\times\Omega}\frac{|u(x)-u(y)|^{p}}{|x-y|^{n+sp}}dx\, dy\right)^{1/p},$$
and the Soblev space
$$W^{s,p}(\Omega):=\{u\in L^{p}(\Omega):\, [u]_{s,p,\Omega}<\infty\},$$
(see for instance \cite{adams, guida}). We will also denote
$$L_{sp}^{p-1}(\mathbb{R}^{n}):=\{u\in L^{p-1}_{loc}(\mathbb{R}^{n}):\, \int_{\R^n}\frac{|u(x)|^{p-1}}{(1+|x|)^{n+sp}}\, dx<\infty\},$$
the space of functions with appropriate decay at infinity. 
\begin{definition}\label{defWeak}
A function $u\in {W^{s,p}(\Omega)}\cap L^{p-1}_{sp}(\R^n)$ is a weak supersolution (subsolution) of \eqref{problem} if 
$$\int_{\R^n}\int_{\R^n}\frac{|u(x)-u(y)|^{p-2}(u(x)-u(y))(\varphi(x)-\varphi(y))}{|x-y|^{n+sp}}dx\, dy\geq (\leq )\int_\Omega f(x,u,D_s^pu)\varphi dx,$$
for all non-negative $\varphi\in C_0^\infty(\Omega)$. We say that $u$ is a weak solution of \eqref{problem} if it is both a weak supersolution and subsolution to the problem.
\end{definition} 
Notice that the double integral appearing in this definition can be written in the domain 
$$Q_{K}:=(K\times\mathbb{R}^{n})\cup((\mathbb{R}^{n}\setminus K)\times K)=\R^{2n}\setminus(\mathbb{R}^{n}\setminus K)^{2},$$ 
where $K:=\mbox{supp}(\varphi)$. Without loss of generality, along the work we will write indistinctly the double integral in $Q_{K}$ or in the whole space $\mathbb{R}^{2n}$ when there is no need to specify. 
\medskip

It is well known that the classical $p$-Laplacian presents additional difficulties when $1<p<2$, since div$(|\nabla u|^{p-2}\nabla u)$ is not well defined at the points where $\nabla u$ vanishes, and any valid notion of viscosity solution has to deal with this issue. The same problem holds here. In the nonlocal setting the singular range corresponds to $1<p\leq \frac{2}{2-s}$, case in which the integral in \eqref{operador} is not well defined even for smooth functions. To give pointwise sense to the operator we need to work with a more restricted class of functions (see \cite[Lemma 3.7]{KKL}). Given an open set $\widetilde\Omega\subset\Omega$, we define
$$
C^{2}_{\beta}(\widetilde\Omega):=\{u\in C^{2}(\widetilde\Omega):\, \sup_{x\in\widetilde\Omega}\left(\frac{\min\{\dist (x, N_{u}),1\}^{\beta-1}}{|\nabla u(x)|}+\frac{|D^2 u(x)|}{(\dist(x, N_{u}))^{\beta-2}}\right)<+\infty\},
$$
where $N_u$ is the set of critical points associated to the regular function $u$, that is,
$$N_{u}:=\{x\in\Omega:\, \nabla u(x)=0\}.$$
In the spirit of \cite{KKL}, we can give the definition of viscosity solution.
\begin{definition}\label{defVisc}
A function $u:\R^n\rightarrow [-\infty,+\infty]$ is a viscosity supersolution (subsolution) of \eqref{problem} in $\Omega$ if:
\begin{itemize}
\item[(i)] $u<+\infty$ ($u>-\infty$) a.e. in $\R^n$, $u>-\infty$ ($u<+\infty$) a.e. in $\Omega$.
\item[(ii)] $u$ is lower (upper) semicontinuous in $\Omega$.
\item[(iii)] If $\psi \in C^2(B_r(x_0)){\cap L^{p-1}_{sp}(\R^n)}$ for some $B_r(x_0)\subset \Omega$ such that $\psi(x_0)=u(x_0)$, $\psi\leq u$ ($\psi\geq u$) in $\R^n$ and one of the following conditions hold:
\begin{itemize}
\item[(a)] $p>\frac{2}{2-s}$ or $\nabla \psi(x_0)\neq 0$,
\item[(b)] $1<p\leq \frac{2}{2-s}$, $\nabla \psi(x_0)=0$ such that $x_0$ is an isolated critical point of $\psi$, and $\psi \in C_\beta^2(B_r(x_0))$ for some $\beta>\frac{sp}{p-1}$,
\end{itemize}
then
$${(-\Delta)^s_p\psi(x_0)\geq (\leq) \,f(x_0,\psi(x_0),D_s^p \psi(x_0))}.$$
\item[(iv)] $u_-:=\max\{-u,0\}$ ($u_+:=\max\{u,0\}$) belongs to $L^{p-1}_{sp}(\R^n)$.
\end{itemize}
If it is at the same time viscosity super and subsolution we say that $u$ is a viscosity solution to \eqref{problem}.
\end{definition}

\begin{remark}
{In many applications the previous definition can be stated in a slightly more general form by allowing the test function to be $\psi\leq u$ only in $B_r(x_0)$ and $\psi=u$ in $\R^n\setminus B_r(x_0)$. This would be also our case if the right hand side of \eqref{problem} were $f=f(x,u)$. However, the dependence on $D_s^pu$ makes us need a larger set of test functions.}
\end{remark}
\subsection{Auxiliary results}
In this section we introduce some properties that will be used later in the proof of the main theorems. We first prove an arithmetical identity.
\begin{lemma}\label{lemmaL}
Let $1<p<+\infty$ and
\begin{equation}\label{laL}
L(\gamma):=|\gamma|^{p-2}\gamma,\quad \gamma\in \R.
\end{equation}
Then
\begin{equation}\label{ineq1}
L(a)-L(b)=(p-1)(a-b)\int_0^1|ta+(1-t)b|^{p-2}dt,\quad a,b\in \R.
\end{equation}
Furthermore,
\begin{equation}\label{ineq2}
\int_0^1|ta+(1-t)b|^{p-2}dt\leq 
\begin{cases}
C(|a|^{p-2}+|b|^{p-2})\quad \mbox{ if }p\geq 2,\\
C|a-b|^{p-2}\quad \mbox{ if }1<p<2,
\end{cases}
\end{equation}
where $C=C(p)>0$.
\end{lemma}

\begin{proof}
We follow the ideas in the proof of \cite[Lemma 1]{KKP}. Consider
$$f(t):=|ta+(1-t)b|^{p-2}(ta+(1-t)b).$$
Thus
$$L(a)-L(b)=\int_0^1 f'(t)dt$$
and \eqref{ineq1} follows. 
The case $p\geq 2$ in \eqref{ineq2} is obtained straightforward.  For $1<p<2$ we define
$$g(\alpha):=\int_0^1|t\alpha+1-t|^{p-2}dt,\qquad \alpha\in [-1,1].$$
Notice that $g(1)=1$ and, for $\alpha\in [-1,1)$, making the change of variable $\tau:=t\alpha+1-t$ we have
$$g(\alpha)=\frac{1}{1-\alpha}\int_\alpha^1|\tau|^{p-2}d\tau\leq \frac{2}{p-1}.$$
Without loss of generality we assume {$|a|\leq |b|$}. Thus
\begin{equation*}\begin{split}
\int_0^1|ta+(1-t)b|^{p-2}dt&=|b|^{p-2}\int_0^1|t\frac{a}{b}+1-t|^{p-2}dt=|b|^{p-2}g\left(\frac{a}{b}\right)\\
&\leq \frac{4}{p-1} |a-b|^{p-2},
\end{split}\end{equation*}
since $|a-b|\leq 2|b|$ and $p-2<0$.
\end{proof}

The next result establishes a Caccioppoli type inequality that will be fundamental to prove Theorem \ref{ViscToWeak}. 
\begin{prop}\label{Caccioppoli}
Let {$u\in L^\infty(\R^n)$} be a weak supersolution of \eqref{problem}, and let $f:\Omega\times\R\times\R^n$ continuous and satisfying \eqref{growth}. Then there exists a constant {$C=C(p,K,\phi)>0$} such that
\begin{equation*}\begin{split}
\int_{K}\int_{\R^n}&\frac{|u(x)-u(y)|^p}{|x-y|^{n+sp}}\xi^p(x)dy\, dx\\
&\leq C\left[(\mbox{osc}(u))^p\left(\int_{K}\int_{\R^n}\frac{|\xi(x)-\xi(y)|^p}{|x-y|^{n+sp}}dy\, dx+{\gamma_{\infty,u}^p}\right)+\mbox{osc}(u)\right],
\end{split}\end{equation*}
for all $\xi\in C_0^\infty(\Omega)$, $0\leq \xi\leq 1$, where $\mbox{osc}(u):=\sup_{\R^n}u-\inf_{\R^n}u$, {$\gamma_{\infty,u}:=\max_{t\in(-\|u\|_{L^\infty(\Omega)},\|u\|_{L^\infty(\Omega)})}|\gamma(t)|$} and $K:=\mbox{supp}(\xi)$.
\end{prop}

\begin{proof}
Let us consider $\xi\in C_0^\infty(\Omega)$, $0\leq \xi\leq 1$ with $K:=\mbox{supp}(\xi)$ and 
$$\varphi(x):=\begin{cases}
(\sup_{\R^n} u -u(x))\xi^p(x),\;\;x\in\Omega,\\
0,\;\;x\in \R^n\setminus\Omega.
\end{cases}$$
Since $u$ is a weak supersolution of \eqref{problem}, testing with $\varphi$ and noticing that
\begin{equation*}\begin{split}
(\sup_{\R^n} u &-u(x))\xi^p(x)-(\sup_{\R^n} u -u(y))\xi^p(y)\\
&=-(u(x)-u(y))\xi^p(x)+(\xi^p(x)-\xi^p(y))(\sup_{\R^n}u-u(y)),
\end{split}\end{equation*}
we have
\begin{equation*}\begin{split}
\int_K &f(x,u,D_su)(\sup_{\R^n} u -u)\xi^pdx\leq -\int_K\int_{\R^n}\frac{|u(x)-u(y)|^p}{|x-y|^{n+sp}}\xi^p(x)dy\, dx\\
&+\iint_{Q_K}\frac{|u(x)-u(y)|^{p-2}(u(x)-u(y))(\xi^p(x)-\xi^p(y))(\sup_{\R^n}u-u(y))}{|x-y|^{n+sp}}dx\, dy,
\end{split}\end{equation*}
and thus
\begin{equation*}\begin{split}
\int_K\int_{\R^n}\frac{|u(x)-u(y)|^p}{|x-y|^{n+sp}}\xi^p(x)dy\, dx
\leq &\iint_{Q_K}\frac{|u(x)-u(y)|^{p-1}|\xi^p(x)-\xi^p(y)|\mbox{osc}(u)}{|x-y|^{n+sp}}dx\, dy\\
&-\int_K f(x,u,D_s^pu)(\sup_{\R^n} u -u)\xi^pdx.
\end{split}\end{equation*}
Let $\delta>0$ small to be chosen later.  Since $1<p<+\infty$ we have that
$$|\xi^p(x)-\xi^p(y)|\leq {C(p)} |\xi(x)-\xi(y)|\left(|\xi(x)|^{p-1}+|\xi(y)|^{p-1}\right),$$
then, by applying Young's inequality with exponents $p/p-1$ and $p$ it follows
\begin{equation*}\begin{split}
&\int_K\int_{\R^n}\frac{|u(x)-u(y)|^p}{|x-y|^{n+sp}}\xi^p(x)dy\, dx
\leq \delta {C(p)}\iint_{Q_K}\frac{|u(x)-u(y)|^{p}(|\xi(x)|^p+|\xi(y)|^p)}{|x-y|^{n+sp}}dx\, dy\\
&\;+\delta^{1-p} {C(p)}(\mbox{osc}(u))^p\iint_{Q_K}\frac{|\xi(x)-\xi(y)|^p}{|x-y|^{n+sp}}dx\, dy-\int_K f(x,u,D_s^pu)(\sup_{\R^n} u -u)\xi^pdx.
\end{split}\end{equation*}
Choosing $\delta$ small, by the symmetry of the first term in the r.h.s, and applying \eqref{growth}, we get
\begin{equation}\label{spar}
\begin{split}
\int_K\int_{\R^n}&\frac{|u(x)-u(y)|^p}{|x-y|^{n+sp}}\xi^p(x)dy\, dx
\leq {C(p)}\left[(\mbox{osc}(u))^p\iint_{Q_K}\frac{|\xi(x)-\xi(y)|^p}{|x-y|^{n+sp}}dx\, dy\right.\\
&\quad \left.+{\gamma_{\infty,u}}\int_K |D_s^pu|^{\frac{p-1}{p}}\mbox{osc}(u)\xi^pdx+\|\phi\|_{L^\infty}|K|\mbox{osc}(u)\right],
\end{split}
\end{equation}
where {$\gamma_{\infty,u}:=\max_{t\in(-\|u\|_{L^\infty(\Omega)},\|u\|_{L^\infty(\Omega)})}|\gamma(t)|$}. By Young's inequality again and the fact that $\xi\leq 1$,
\begin{equation*}\begin{split}
{\gamma_{\infty,u}}\int_K &|D_s^pu|^{\frac{p-1}{p}}\mbox{osc}(u)\xi^pdx\leq {\gamma_{\infty,u}}\int_K |D_s^pu|^{{\frac{p-1}{p}}}\mbox{osc}(u)\xi^{p-1}dx\\
&\leq \delta \int_K |D_s^p u|\xi^p dx+\delta^{1-p}{\gamma_{\infty,u}^pC(p,
K)}(\mbox{osc}(u))^p\\
&\leq \delta \int_K\int_{\R^n}\frac{|u(x)-u(y)|^p}{|x-y|^{n+sp}}\xi^p(x)dx\, dy+\delta^{1-p}{\gamma_{\infty,u}^pC(p,
\Omega)}(\mbox{osc}(u))^p.
\end{split}\end{equation*}
Thus, choosing $\delta$ small, from \eqref{spar} we have
\begin{equation*}\begin{split}
\int_K\int_{\R^n}&\frac{|u(x)-u(y)|^p}{|x-y|^{n+sp}}\xi^p(x)dy\, dx
\\
&\leq {C(p,\phi, K)}\left[(\mbox{osc}(u))^p\left(\int_K\int_{\R^n}\frac{|\xi(x)-\xi(y)|^p}{|x-y|^{n+sp}}dy\, dx+{\gamma_{\infty,u}^p}\right)+\mbox{osc}(u)\right].
\end{split}\end{equation*}
\end{proof}
To conclude this section we show the continuity of the operator $D_s^p$ with respect to small regular perturbations, that will be crucial in the proof of Theorem \ref{WeakToVisc}. 
\begin{lemma}\label{Lema1}
Let $r>0$, $x_0\in\mathbb{R}^{n}$ and {$F\in L^{p-1}_{sp}(\mathbb{R}^{n})$, Lipschitz continuous in $B_r(x_0)$}. For every $\varepsilon>0$, ${0<\rho<r}$ and $\eta\in C^{2}_{0}(B_{r}(x_0))$ verifying $0\leq\eta\leq 1$, there exists a constant $\widetilde{\theta}=\widetilde{\theta}(\varepsilon,\rho,\eta)>0$ such that
$$\sup_{B_{\rho}(x_0)}|D_s^{p}F-D_s^{p}(F+\theta\eta)|<\varepsilon\qquad \mbox{ for all }0\leq\theta<\widetilde{\theta}.$$
\end{lemma}
\begin{proof}
Let us fix $\varepsilon>0$ and {$\rho$ such that $0<\rho<r$} and consider {a function $\eta\in C^{2}_{0}(B_{r}(x_0))$ such that $0\leq\eta\leq 1$}. Given $0<\theta<1$, if $x\in B_{\rho}(x_0)$ it follows that
\begin{equation*}
\begin{split}
|D_s^{p}F(x)-&D_s^{p}(F+\theta\eta)(x)|\leq\int_{B_{\delta}(x)}\frac{|F(x)-F(y)|^{p}}{|x-y|^{n+sp}}\, dy\\
&+\int_{B_{\delta}(x)}\frac{|F(x)+\theta\eta(x)-F(y)-\theta\eta(y)|^{p}}{|x-y|^{n+sp}}\, dy
\\
&+\int_{\mathbb{R}^{n}\setminus B_{\delta}(x)}\frac{\left||F(x)-F(y)|^{p}-|F(x)+\theta\eta(x)-F(y)-\theta\eta(y)|^{p}\right|}{|x-y|^{n+sp}}\, dy ,
\end{split}\end{equation*}
where ${0<\delta<r-\rho}$ is a small parameter to be chosen later. By the algebraic inequality (see  \cite[Lemma 3.4]{KKL})
\begin{equation}\label{showman}
||a|^{p}-|b|^{p}|\leq C(p)|a-b|(|b|+|a-b|)^{p-1},\, p>0,\, a,\, b\in\R,
\end{equation}
the regularity of $F$ and the fact that $\eta\leq 1$, it is clear that
\begin{equation*}
\begin{split}
|D_s^{p}F(x)-&D_s^{p}(F+\theta\eta)(x)|\leq C\left({(1+\theta^p)}\int_{B_{\delta}(x)}\frac{|x-y|^{p}}{|x-y|^{n+sp}}\, dy\right.\\
&+\left.{\theta}\int_{\mathbb{R}^{n}\setminus B_{\delta}(x)}\frac{({4}\theta+|F(x)-F(y)|)^{p-1}}{|x-y|^{n+sp}}\, dy\right),
\end{split}
\end{equation*}
where $C=C(p,{\sup_{B_{\delta+\rho}(x_0)}\{|\nabla F|, |\nabla\eta|\}})$. 
Thus, since $\theta<1$ and $p>1$, it follows that
\begin{equation}\label{enem}
\begin{split}
|D_s^{p}F(x)-&D_s^{p}(F+\theta\eta)(x)|\leq C\delta^{p(1-s)}+C{\theta}\bigg(\theta^{p-1}\delta^{-sp}+\|F\|^{p-1}_{L^{\infty}(B_{\rho}(x_0))}\delta^{-sp}\\
&+\int_{\mathbb{R}^{n}\setminus B_{\delta}(x)}\frac{|F(y)|^{p-1}}{|x-y|^{n+sp}}\, dy\bigg)\\
&\leq C\left({\delta^{p(1-s)}}+{\theta}\delta^{-sp}\right),
\end{split}
\end{equation}
due to the fact that $F\in L^{p-1}_{sp}(\mathbb{R}^{n})$.
Choosing first $\delta=\delta(\eta)$ and then $\theta=\theta(\delta)$ appropriately small we have
$$|D_s^{p}F(x)-D_s^{p}(F+\theta\eta)(x)|\leq \varepsilon,$$
and the conclusion follows from \eqref{enem} by taking the supremum over $x\in B_{\rho}(x_0)$.
\end{proof}
\section{Viscosity solutions are weak solutions}
The aim of this section is to prove Theorem \ref{ViscToWeak}. 
\subsection{Infimal convolution}
Let $\varepsilon>0$. We define the infimal convolution of a function $u:\R^n\to\R$ as
\begin{equation}\label{infConv}
{u_\varepsilon(x):=
\inf_{y\in \R^n}\left(u(y)+\frac{|x-y|^q}{q\varepsilon^{q-1}}\right)},
\end{equation}
where 
$$q=2\;\; \mbox{ if }\;\; p>\frac{2}{2-s},\qquad\quad  q>\frac{sp}{p-1}\geq 2\;\;\mbox{ if }\;\; 1<p\leq \frac{2}{2-s}.$$
The inf (and sup) convolutions, that originally appeared in convex analysis, are very useful in the viscosity theory due to their nice approximation properties and because they are {\it nearly} $C^2$ functions. In the following lemma, whose proof can be found, for instance, in the appendix of \cite{JJ}, we summarize some fundamental properties of the infimal convolutions.
\begin{lemma}\label{propInf}Suppose $u$ is bounded and lower semicontinuous in $\R^n$. Then:
\begin{itemize}
\item[(i)] There exists $r(\varepsilon)>0$ such that
$${u_\varepsilon(x)=\inf_{y\in B_{r(\varepsilon)}(x)}\left(u(y)+\frac{|x-y|^q}{q\varepsilon^{q-1}}\right)},$$
where $r(\varepsilon)\to 0$ as $\varepsilon\to 0$.

\item[(ii)] The sequence $u_\varepsilon$ is increasing {as $\varepsilon\to 0$} and $u_\varepsilon\to u$ pointwise in {$\R^n$}.

\item[(iii)] $u_\varepsilon$ is locally Lipschitz continuous and twice differentiable a.e. Actually, for almost every $x,y\in \R^n$,
$$u_\varepsilon(y)=u_\varepsilon(x)+\nabla u_\varepsilon(x)\cdot (x-y)+\frac 12 D^2u_\varepsilon(x)(x-y)^2+o(|x-y|^2).$$

\item[(iv)] $u_\varepsilon$ is semiconcave, that is, there exists a constant $C>0$, that depends on $q$, $\varepsilon$ and the oscillation $\sup_{\R^n} u-\inf_{\R^n} u$, such that the function
$$x\mapsto u_\varepsilon(x)-C|x|^2,\qquad x\in \R^n,$$
is concave. In particular 
$$D^2u_\varepsilon(x)\leq 2CI\quad \mbox{ for a.e. }x\in \R^n.$$

\item[(v)] The set
$$Y_\varepsilon(x):=\left\{y\in B_{r(\varepsilon)}(x): \, u_\varepsilon(x)=u(y)+\frac{|x-y|^q}{q\varepsilon^{q-1}}\right\},$$
is non empty and closed for every $x\in \R^n$.

\item[(vi)] {If $\nabla u_\varepsilon(x)=0$ then $u_\varepsilon(x)=u(x)$.}
\end{itemize}
\end{lemma}
\begin{remark} \label{InfBound}
{Since $u_\varepsilon$ is increasing as $\varepsilon\to 0$, by definition 
$$-\|u\|_{L^\infty(\R^n)}\leq \inf_{y\in\R^n}u(y)\leq u_\varepsilon(x)\leq u(x),$$
for every $x\in\R^n$. Hence, if $u\in L^\infty(\R^n)$ then $\|u_\varepsilon\|_{L^\infty(\R^n)}$ is uniformly bounded.}
\end{remark}

The strategy to prove Theorem \ref{ViscToWeak} is the following: first finding out the problem satisfied in viscosity (and pointwise) sense by $u_\varepsilon$; then, taking advantage of the higher regularity and the specific properties of $u_\varepsilon$, passing from the pointwise to the weak formulation; finally passing to the limit to obtain the weak problem satisfied by $u$.

We start by identifying the problem fulfilled by $u_\varepsilon$.

\begin{lemma}\label{lemma1}
{Let $u:\R^n\rightarrow \R$ bounded and lower semicontinuous in {$\R^n$}, and let $f=f(x,t,\eta)$ be continuous in $\Omega\times\R\times\R^n$ and non increasing in $t$.} For any $1<p<+\infty$, if $u$ is a viscosity supersolution of
$$(-\Delta)^s_pu=f(x,u,D_s^pu) \quad \mbox{ in }\Omega,$$
then $u_\varepsilon$ is a viscosity supersolution of 
\begin{equation}\label{fin}
(-\Delta)^s_p v=f_\varepsilon(x,v,D_s^p v) \quad \mbox{ in }\Omega_{r(\varepsilon)},\end{equation}
where $u_{\varepsilon}$ was defined in \eqref{infConv}, $\Omega_{r(\varepsilon)}:=\{x\in\Omega:\mbox{dist}(x,\partial\Omega)>r(\varepsilon)\}$ and 
\begin{equation}\label{f_eps}
f_\varepsilon(x,t,\eta):=\inf_{y\in {B_{r(\varepsilon)}(x)}}f(y,t,\eta). 
\end{equation}
Furthermore, 
\begin{equation}\label{point}
(-\Delta)^s_pu_\varepsilon(x)\geq f_\varepsilon(x,u_\varepsilon(x),D_s^pu_\varepsilon(x)) \quad \mbox{ a.e. }  x\in\Omega_{r(\varepsilon)}.
\end{equation}
\end{lemma}

\begin{proof}
The first part follows as in \cite[Lemma 2.1]{MO}. Let us define, for every $z\in B_{r(\varepsilon)}(0)$,
$$\phi_z(x):=u(z+x)+\frac{|z|^q}{q\varepsilon^{q-1}},\quad x\in \R^n.$$
We claim that, for every $z\in B_{r(\varepsilon)}(0)$, this function is a viscosity supersolution of \eqref{fin}. Indeed let $x_0\in \Omega_{r(\varepsilon)}$ and $\varphi$ satisfying the hypotheses in Definition \ref{defVisc}, with $\varphi(x_0)=\phi_z(x_0)$, $\varphi\leq \phi_z$, and denote $y:=z+x$, $y_0=z+x_0\in B_{r(\varepsilon)}(x_0)\subseteq\Omega$. Thus the function
$$\tilde{\varphi}(y):=\varphi(y-z)-\frac{|z|^q}{q\varepsilon^{q-1}}$$
satisfies $\tilde{\varphi}(y_0)=u(y_0)$, $\tilde{\varphi}\leq u$, and therefore, since $u$ is a viscosity supersolution in $\Omega$ and $f$ is non increasing in the second variable,
\begin{equation*}\begin{split}
(-\Delta)^s_p\varphi(x_0)=(-\Delta)_p^s\tilde{\varphi}(y_0)&\geq f(y_0,\tilde{\varphi}(y_0),D_s^p\tilde{\varphi}(y_0))\\
&=f(z+x_0,\varphi(x_0)-\frac{|z|^q}{q\varepsilon^{q-1}},D_s^p\varphi(x_0))\\
&\geq f_\varepsilon(x_0),\varphi(x_0),D_s^p\varphi(x_0)),
\end{split}\end{equation*}
and the claim follows.

Let now $\psi$ be a function satisfying the regularity conditions in Definition \ref{defVisc} and such that $\psi(x_0)=u_\varepsilon(x_0)$, $\psi\leq u_\varepsilon$. By (i)  and (v) in Lemma \ref{propInf}, we can write
$$u_\varepsilon(x)=\inf_{y\in B_{r(\varepsilon)}(x)}\left(u(y)+\frac{|x-y|^q}{q\varepsilon^{q-1}}\right)=\inf_{z\in B_{r(\varepsilon)}(0)}\phi_z(x),\quad x\in\mathbb{R}^{n},$$
and we know that there exists $\hat{z}\in B_{r(\varepsilon)}(0)$ such that $u_\varepsilon(x_0)=\phi_{\hat{z}}(x_0)$. Since by definition $\psi\leq u_\varepsilon\leq \phi_{\hat{z}}$, $\psi$ may be used as a test function in the problem satisfied by $\phi_{\hat{z}}$ to conclude
$$(-\Delta)^s_p\psi(x_0)\geq f_\varepsilon(x_0,\psi(x_0),D_s^p\psi(x_0)).$$
Therefore $u_\varepsilon$ is a viscosity supersolution of \eqref{fin}.

Let us prove \eqref{point}. Fix $x\in\Omega_{r(\varepsilon)}$ such that $u_\varepsilon$ is twice differentiable at $x$ and $r>0$ such that $B_r(x)\subset\Omega_{r(\varepsilon)}$, and assume first that $p>\frac{2}{2-s}$ or $\nabla u_\varepsilon(x)\neq 0$.  Define
$$\psi_{\delta}(y):=u_\varepsilon(x)+\nabla u_\varepsilon(x)\cdot (x-y)+\frac 12 \left(D^2u_\varepsilon(x)-\delta I\right)(x-y)^2,$$
with $\delta>0$ and $I$ the identity matrix.
Notice that $\psi_\delta(x)=u_\varepsilon(x)$ and $\psi_\delta\in C^2(B_r(x))$. Consider now the function
$$\psi_r(y):=\begin{cases}
\psi_{\delta}(y),\quad y\in B_r(x),\\
u_\varepsilon(y),\quad y\in \R^n\setminus B_r(x),
\end{cases}$$
that, for $\delta$ large enough, satisfies $\psi_r\leq u_\varepsilon$. Thus, using $\psi_r$ as a test function in the viscosity sense in the problem satisfied by $u_\varepsilon$, by \cite[Lemma 3.6]{KKL} we obtain
\begin{equation}\begin{split}
\int_{\R^n\setminus B_r(x)}&\frac{|u_\varepsilon(x)-u_\varepsilon(y)|^{p-2}(u_\varepsilon(x)-u_\varepsilon(y))}{|x-y|^{n+sp}}dy\\
&= \int_{\R^n\setminus B_r(x)}\frac{|\psi_r(x)-\psi_r(y)|^{p-2}(\psi_r(x)-\psi_r(y))}{|x-y|^{n+sp}}dy\\
&\geq (-\Delta)_s^p\psi_r(x)-o_r(1)\\
&\geq f_\varepsilon(x,u_\varepsilon(x), D_s^p\psi_r(x))-o_r(1),\label{pointueps}
\end{split}\end{equation}
where $o_r(1)\to 0$ as $r\to 0$. Likewise 
\begin{equation*}\begin{split}
D_s^p\psi_r(x)&={\int_{\R^n}\frac{|\psi_r(x)-\psi_r(y)|^{p}}{|x-y|^{n+sp}}dy}\\
&={\int_{\R^n\setminus B_r(x)}\frac{|u_\varepsilon(x)-u_\varepsilon(y)|^{p}}{|x-y|^{n+sp}}dy+\int_{B_r(x)}\frac{|\psi_{\delta}(x)-\psi_{\delta}(y)|^{p}}{|x-y|^{n+sp}}dy}\\
&={\int_{\R^n\setminus B_r(x)}\frac{|u_\varepsilon(x)-u_\varepsilon(y)|^{p}}{|x-y|^{n+sp}}dy}+O(r^{p(1-s)}),
\end{split}\end{equation*}
and hence
$$\lim_{r\to 0}D_s^p\psi_r(x)= D_s^pu_\varepsilon(x).$$
Passing to the limit as $r\to 0$ in \eqref{pointueps}, taking into account the definition, via principal value, of $(-\Delta)^{s}_p$, we obtain \eqref{point} if $p>\frac{2}{2-s}$ or $\nabla u_\varepsilon(x)\neq 0$.

Consider now the case $1<p\leq \frac{2}{2-s}$ and $\nabla u_\varepsilon(x)=0$. Here, instead of testing in the problem satisfied by $u_\varepsilon$, we will need to use the fact that $u$ is a viscosity supersolution. By (vi) of Lemma \ref{propInf} we know that
$$u(x)=u_{\varepsilon}(x)\leq u(y)+\frac{|x-y|^q}{q\varepsilon^{q-1}}\quad\mbox{ for all }y\in\R^n,\quad q>\frac{sp}{p-1}\geq 2.$$
Thus, if we define
$$\zeta_r(y):=\begin{cases}
\zeta(y),\quad y\in B_r(x),\\
u_\varepsilon(y),\quad y\in\R^n\setminus B_r(x),
\end{cases}\quad \mbox{ with }\quad \zeta(y):=u(x)-\frac{|x-y|^q}{q\varepsilon^{q-1}},$$
since it can be checked that $\zeta_r\in C^2_q(B_r(x))$, and clearly $\zeta_r(x)=u(x)$, $\zeta_r\leq u$,  we can use it as test function in the problem satisfied by $u$ to obtain
\begin{equation}\label{chaturanga}
\begin{split}
(-\Delta)^s_p\zeta_r(x)\geq f(x,\zeta_r(x),D_s^p\zeta_r(x))\geq f_\varepsilon (x, u_\varepsilon(x),D_s^p\zeta_r(x)).
\end{split}\end{equation}
Noticing that, by \cite[Lemma 3.7]{KKL},
$$(-\Delta)^s_p\zeta_r(x)\leq\int_{\R^n\setminus B_r(x)}\frac{|u_\varepsilon(x)-u_\varepsilon(y)|^{p-2}(u_\varepsilon(x)-u_\varepsilon(y))}{|x-y|^{n+sp}}dy+o_r(1),$$
and
$$D_s^p\zeta_r(x)= \int_{\R^n\setminus B_r(x)}\frac{|u_\varepsilon(x)-u_\varepsilon(y)|^{p}}{|x-y|^{n+sp}}dy+o_r(1),$$
by passing to the limit in \eqref{chaturanga}, as $r\to 0$, we also obtain \eqref{point} when $x$ is a critical point of $u_{\varepsilon}$ and $1<p\leq \frac{2}{2-s}$.
\end{proof}

In the next two lemmas we pass from the pointwise sense of $(-\Delta)^s_p u_\varepsilon$ to a weak formulation. The infimal convolutions provide us useful information on the way we approach the function $u$, but they are not smooth enough (they are not $C^2$ nor $C^2_q$) to integrate by parts. Thus we need to regularize $u_\varepsilon$ via convolution with a standard mollifier, integrating by parts, and then passing to the limit. This last step is extremely tricky since we do not have a uniform estimate of $\|D^2u_\varepsilon\|_{L^\infty}$, meaning $D^2u_\varepsilon$ the Hessian matrix of $u_\varepsilon$. We only know $D^2u_\varepsilon\leq 2CI$ a.e. in $\Omega_{r(\varepsilon)}$ (see Lemma \ref{propInf}), but nothing about the size when the second derivative is negative. This obliges to control the sign of the quotients, a rather delicate issue due to the nonlocality and the nonlinearity of the operator. 

\begin{lemma}\label{lemma2}
Assume $p>\frac{2}{2-s}$. For all $\psi\in C_0^\infty(\Omega_{r(\varepsilon)})$, $\psi\gneq 0$,
$$\iint_{Q_{K}}\frac{|u_\varepsilon(x)-u_\varepsilon(y)|^{p-2}(u_\varepsilon(x)-u_\varepsilon(y))(\psi(x)-\psi(y))}{|x-y|^{n+sp}}dx\, dy\geq \int_{K} ((-\Delta)_p^su_\varepsilon)\psi dx,$$
where $u_{\varepsilon}$ was given in \eqref{infConv}, with $u\in L^{p-1}_{sp}(\R^n)\cap L^{\infty}_{loc}(\R^n)$ and $K:=\mbox{supp}(\psi)$.
\end{lemma}

\begin{proof}
We define
\begin{equation}\label{uepsdelta}
u_{\varepsilon,\delta}(x):=\begin{cases}
\displaystyle\int_{B_\delta(0)}\eta_\delta(y)u_{\varepsilon}(x-y)dy,\quad x\in \Omega_{r(\varepsilon)},\\
{u_\varepsilon(x)},\quad x\in \R^n\setminus\Omega_{r(\varepsilon)}.
\end{cases}
\end{equation}
with
$$\eta_\delta(y):=\delta^{-n}\eta\left(\frac{y}{\delta}\right),$$
where $\eta$ is the standard mollifier. We observe that $u_{\varepsilon,\delta}$ is a sequence of smooth semiconcave functions converging a.e. to $u_{\varepsilon}$. 
Since $u_{\varepsilon,\delta} \in C^2(\Omega_{r(\varepsilon)})\cap {L_{sp}^{p-1}(\R^n)}$ it can be seen that
\begin{equation*}\begin{split}
\iint_{Q_K}\frac{|u_{\varepsilon,\delta}(x)-u_{\varepsilon,\delta}(y)|^{p-2}(u_{\varepsilon,\delta}(x)-u_{\varepsilon,\delta}(y))(\psi(x)-\psi(y))}{|x-y|^{n+sp}}&dx\, dy\\
= \int_K & ((-\Delta)_p^su_{\varepsilon,\delta})\psi dx,
\end{split}\end{equation*}
for every $\psi\in C_0^\infty (\Omega_{r(\varepsilon)})$ with $K:=\mbox{supp}(\psi)$. The goal is to pass to the limit as $\delta\to 0$ to conclude the result.

\vspace{2pt}
\noindent STEP 1: Proving that
\begin{equation}\begin{split}\label{step1}
\lim_{\delta\to 0}&\iint_{Q_K}\frac{|u_{\varepsilon,\delta}(x)-u_{\varepsilon,\delta}(y)|^{p-2}(u_{\varepsilon,\delta}(x)-u_{\varepsilon,\delta}(y))(\psi(x)-\psi(y))}{|x-y|^{n+sp}}dx\, dy\\
& \qquad =\iint_{Q_K}\frac{|u_\varepsilon(x)-u_\varepsilon(y)|^{p-2}(u_\varepsilon(x)-u_\varepsilon(y))(\psi(x)-\psi(y))}{|x-y|^{n+sp}}dx\, dy.
\end{split}\end{equation}
Let us define
$$F_\delta(x,y):=\frac{|u_{\varepsilon,\delta}(x)-u_{\varepsilon,\delta}(y)|^{p-2}(u_{\varepsilon,\delta}(x)-u_{\varepsilon,\delta}(y))(\psi(x)-\psi(y))}{|x-y|^{n+sp}},\quad (x,y)\in Q_{K}.$$
Since by symmetry 
$$\iint_{Q_K} F_\delta(x,y)\, dx\,dy=\left(2\iint_{K\times(\R^{n}\setminus K)}+\iint_{K\times K}\right)F_\delta(x,y)\, dx\,dy,$$
to obtain \eqref{step1}, by the dominated convergence theorem, it is enough to prove that
\begin{equation}\label{DCTcond}
|F_\delta(x,y)|\leq F(x,y),
\end{equation}
where $F(x,y)$ is a function in $L^1(K\times\R^n)$, independent of $\delta$. {Let us define $$r:=\mbox{dist}(K,\partial\Omega_{r(\varepsilon)})>0\quad \mbox{ and }\quad K_{r/2}:=\{y\in\Omega_{r(\varepsilon)}:\,\mbox{dist}(y,K)\leq \frac{r}{2}\},$$ 
and consider first the case $x\in K$, $y\in K_{r/2}$}. By Young's inequality
\begin{equation*}\begin{split}
|F_{\delta}(x,y)|&\leq \frac{|u_{\varepsilon,\delta}(x)-u_{\varepsilon,\delta}(y)|^{p-1}|\psi(x)-\psi(y)|}{|x-y|^{n+sp}}\\
&\leq C\left(\frac{|u_{\varepsilon,\delta}(x)-u_{\varepsilon,\delta}(y)|^{p}}{|x-y|^{n+sp}}+\frac{|\psi(x)-\psi(y)|^p}{|x-y|^{n+sp}}\right),
\end{split}\end{equation*}
{with $C$ independent of $\delta$} so that, since $u_\varepsilon$ is locally Lipschitz, making $\delta$ small enough we have
\begin{equation*}\begin{split}
|u_{\varepsilon,\delta}(x)-u_{\varepsilon,\delta}(y)|^{p}&=\bigg|\int_{B_\delta(0)}(u_\varepsilon(x-z)-u_\varepsilon(y-z))\eta_\delta(z) dz\bigg|^{p}\\
&\leq C|x-y|^p,
\end{split}\end{equation*}
where $C>0$ depends on the Lipschitz constant of $u_\varepsilon$ in $K_{3r/4}$ but it is independent of $\delta$ since $\|\eta_{\delta}\|_{L^{1}(\R^n)}=1$. Thus we conclude that
$$|F_{\delta}(x,y)|\leq C|x-y|^{p-(n+sp)}\in L^1(K\times {K_{r/2}}),$$
where the right hand side is independent of $\delta$, so clearly \eqref{DCTcond} holds for every $(x,y)\in K\times {K_{r/2}}$.

Consider now the case $(x,y)\in K\times (\R^n\setminus {K_{r/2}})$. Notice that in this case the kernel is not singular anymore since
$$|x-y|\geq {\frac{r}{2}},$$
and thus we are only concerned with the integrability at infinity. Hence, using that $u_\varepsilon$ is locally  bounded (and thus $u_{\varepsilon,\delta}$),
\begin{equation}\label{nonumerada}
|F_\delta(x,y)|\leq C\frac{1+|u_\varepsilon(y)|^{p-1}}{|x-y|^{n+sp}},
\end{equation}
with $C$ independent of $\delta$, and the right hand side belongs to $L^1(K\times (\R^n\setminus{K_{r/2}}))$ due to the fact that {$u_\varepsilon\in L^{p-1}_{sp}(\R^n)$.} Therefore \eqref{DCTcond} and \eqref{step1} hold.

\vspace{2pt}
\noindent STEP 2: $\displaystyle \liminf_{\delta\to 0}\int_K((-\Delta)^s_pu_{\varepsilon,\delta})\psi\, dx\geq\int_K \liminf_{\delta\to 0}((-\Delta)^s_pu_{\varepsilon,\delta})\psi\, dx.$

\vspace{5pt}

We will see that $(-\Delta)^s_p u_{\varepsilon,\delta}(x)\geq -C$, $x\in K$, being $C$ a positive constant independent of $\delta$ and $x$, and then the claim will follow as a consequence of Fatou's lemma. 

Let $r>0$ small so that $r<\mbox{dist}{(K,\partial\Omega_{r(\varepsilon)})}$. Thus $B_r(x)\subset \Omega_{r(\varepsilon)}$ and for $y\in \R^n\setminus B_r(x)$ there holds
$$|x-y|>r.$$
As in STEP 1, since $u_\varepsilon$ is locally bounded,
\begin{equation}\begin{split}
&\bigg|\int_{\R^n\setminus B_r(x)}\frac{|u_{\varepsilon,\delta}(x)-u_{\varepsilon,\delta}(y)|^{p-2}(u_{\varepsilon,\delta}(x)-u_{\varepsilon,\delta}(y))}{|x-y|^{n+sp}}dy\bigg|\\
&\qquad\qquad \qquad\qquad \qquad\qquad\qquad \leq C\int_{\R^n\setminus B_r(x)}\frac{1+|u_\varepsilon(y)|^{p-1}}{|x-y|^{n+sp}}dy,\label{fuera}
\end{split}\end{equation}
which is uniformly bounded as a consequence of $u_\varepsilon\in L^{p-1}_{sp}(\R^n)$.

We now estimate the integral, in the principal value sense, in the ball $B_r(x)$. First, notice that by definition {and Lemma \ref{propInf}}
\begin{equation}\label{boundD2}
D^2u_{\varepsilon,\delta}\leq CI \mbox{ in }{\Omega_{r(\varepsilon)}},
\end{equation}
with $C>0$ independent of $\delta$.  Using the principal value and the odd symmetry it is clear that 
\begin{equation}\begin{split}\label{AB}
I_{B_r(x)}&:={\rm p.v} \int_{B_r(x)}\frac{|u_{\varepsilon,\delta}(x)-u_{\varepsilon,\delta}(y)|^{p-2}(u_{\varepsilon,\delta}(x)-u_{\varepsilon,\delta}(y))}{|x-y|^{n+sp}}dy\\
&=\int_{B_r(x)}\frac{L(u_{\varepsilon,\delta}(x)-u_{\varepsilon,\delta}(y))-L(-\nabla u_{\varepsilon,\delta}(x)\cdot (y-x))}{|x-y|^{n+sp}}dy,
\end{split}\end{equation}
with $L$ defined in \eqref{laL}. By Lemma \ref{lemmaL} with
\begin{equation}\label{aYb}
a:=u_{\varepsilon,\delta}(x)-u_{\varepsilon,\delta}(y),\quad b:=-\nabla u_{\varepsilon,\delta}(x)\cdot (y-x),
\end{equation}
expanding $u_{\varepsilon,\delta}(y)$ we obtain that
\begin{equation}\begin{split}
I_{B_r(x)}&=(p-1)\int_{B_r(x)}\frac{-D^2u_{\varepsilon,\delta}(\eta)(y-x)^2}{|x-y|^{n+sp}}\left(\int_0^1|ta-(1-t)b|^{p-2}dt\right)dy\\
&\geq (p-1)\int_{B_r(x)^{+}}\frac{-D^2u_{\varepsilon,\delta}(\eta)(y-x)^2}{|x-y|^{n+sp}}\left(\int_0^1|ta-(1-t)b|^{p-2}dt\right)dy,\label{B}
\end{split}\end{equation}
where $\eta\in B_r(x)\subset\Omega_{r(\varepsilon)}$ depends on $y$ and
$$B_r(x)^{+}:=B_r(x)\cap\{y: D^{2}u_{\varepsilon, \delta}(\eta)\geq 0\}.$$
If $p\geq 2$ by using that $u_\varepsilon$ is locally Lipschitz, also it is $u_{\varepsilon,\delta}$ and
$$\|\nabla u_{\varepsilon,\delta}\|_{L^\infty(K)}\leq C,$$
with $C$ independent of $\delta$. It follows that
\begin{equation}\label{p_grande0}\begin{split}
0\leq \int_0^1|ta-(1-t)b|^{p-2}dt\leq C|x-y|^{p-2}.
\end{split}\end{equation}
Thus, replacing in \eqref{B} and using \eqref{boundD2}, we get that
\begin{equation}\label{p_grande}
I_{B_r(x)}\geq -C\int_{B_r(x)}\frac{|x-y|^p}{|x-y|^{n+sp}}dy\geq -C,\quad p\geq 2,
\end{equation}
uniformly in $\delta$.

In the case $\frac{2}{2-s}<p<2$, applying again Lemma \ref{lemmaL}, for every $y\in B_r(x)^{+}$ we have that
\begin{equation}\label{p_chico0}\begin{split}
0&\leq \int_0^1|ta-(1-t)b|^{p-2}dt\leq C|a-b|^{p-2}\\
&=|-D^2u_{\varepsilon,\delta}(\eta)(y-x)^2|^{p-2}
=(D^{2}u_{\varepsilon,\delta}(\eta)(y-x)^2)^{p-2},\quad \eta\in B_r(x)\subset\Omega_{r(\varepsilon)}.
\end{split}\end{equation}
Replacing in \eqref{B} and using \eqref{boundD2} we obtain
\begin{equation}\label{p_chico}
I_{B_r(x)}\geq -C\int_{B_r(x)}\frac{|x-y|^{{2(p-1)}}}{|x-y|^{n+sp}}dy\geq -C,\quad \frac{2}{2-s}<p<2.
\end{equation}
By \eqref{p_grande} and \eqref{p_chico} we conclude that 
$$(-\Delta)^s_p u_{\varepsilon,\delta}\geq -C \mbox{ in }K,$$
{with $C>0$ independent of $\delta$} and so, applying Fatou's lemma, we obtain the claim in STEP 2.

\noindent STEP 3: $\displaystyle \int_K \liminf_{\delta\to 0}((-\Delta)^s_pu_{\varepsilon,\delta})\psi\, dx \geq  \int_K ((-\Delta)^s_pu_{\varepsilon})\psi\, dx.$

Reasoning as in Step 2,
\begin{equation*}\begin{split}
(-\Delta)^{s}_{p}u_{\varepsilon,\delta}(x)&={\int_{\R^n\setminus B_{r}(x)}}\frac{|u_{\varepsilon,\delta}(x)-u_{\varepsilon,\delta}(y)|^{p-2}(u_{\varepsilon,\delta}(x)-u_{\varepsilon,\delta}(y))}{|x-y|^{n+sp}}dy\\
&\quad+\int_{B_{r}(x)}\frac{L(u_{\varepsilon,\delta}(x)-u_{\varepsilon,\delta}(y))-L(-\nabla u_{\varepsilon,\delta}(x)\cdot (y-x))}{|x-y|^{n+sp}}dy,
\end{split}\end{equation*}
where $L$ was given in \eqref{laL}. As we saw in  
\eqref{fuera} one has 
\begin{equation}\label{go1}
\frac{|u_{\varepsilon,\delta}(x)-u_{\varepsilon,\delta}(y)|^{p-2}(u_{\varepsilon,\delta}(x)-u_{\varepsilon,\delta}(y))}{|x-y|^{n+sp}}\geq -C\frac{1+|u_\varepsilon(y)|^{p-1}}{|x-y|^{n+sp}}.
\end{equation}
{Using Lemma \ref{lemmaL} and \eqref{p_grande0} and \eqref{p_chico0}} we also obtain
\begin{equation}\begin{split}
&\frac{L(u_{\varepsilon,\delta}(x)-u_{\varepsilon,\delta}(y))-L(-\nabla u_{\varepsilon,\delta}(x)\cdot (y-x))}{|x-y|^{n+sp}}\geq -C \mathcal{F}(x,y),\label{go2}
\end{split}\end{equation}
with
$$\mathcal{F}(x,y):=\begin{cases}
|x-y|^{p-n-sp}\quad p\geq 2,\\
|x-y|^{2(p-1)-n-sp},\quad \frac{2}{2-s}<p<2.
\end{cases}$$
Therefore, since for every $x\in K$,
$$\frac{1+|u_\varepsilon(y)|^{p-1}}{|x-y|^{n+sp}}\in L^{1}({\R^n\setminus B_{r}(x)}),\mbox{ and } \mathcal{F}(x,y)\in L^{1}(B_r(x)),$$
by \eqref{go1} and \eqref{go2} using a generalization of Fatou's Lemma we conclude that
$$\liminf_{\delta\to 0}(-\Delta)^s_pu_{\varepsilon,\delta}\geq (-\Delta)^s_pu_{\varepsilon}.$$
This proves STEP 3 and the whole proof.
\end{proof}

In the case $1<p\leq \frac{2}{2-s}$ it is not enough to work with the regularized versions of $u_\varepsilon$. Even if they are $C^\infty$, the operator is not well defined for functions which are not in $C^2_q$, $q>\frac{sp}{p-1}$, and thus we cannot integrate by parts directly. We overcome this difficuly by regularizing the operator before the integration by parts.
\begin{lemma}\label{lemma22}
Assume $1<p\leq \frac{2}{2-s}$. For all $\psi\in C_0^\infty(\Omega_{r(\varepsilon)})$, $\psi\gneq 0$,
\begin{equation}\label{lemma33}
\iint_{Q_{K}}\frac{|u_\varepsilon(x)-u_\varepsilon(y)|^{p-2}(u_\varepsilon(x)-u_\varepsilon(y))(\psi(x)-\psi(y))}{|x-y|^{n+sp}}dx\, dy\geq \int_{K} ((-\Delta)_p^su_\varepsilon)\psi dx,
\end{equation}
where $u_{\varepsilon}$ was given in \eqref{infConv}, with $u\in L^{p-1}_{sp}(\R^n)\cap L^{\infty}_{loc}(\R^n)$ and $K:=\mbox{supp}(\psi)$.

\end{lemma}

\begin{proof}
We need to regularize the operator to integrate by parts. Indeed, given $\rho>0$ and $\psi\in C_0^\infty(\Omega_{r(\varepsilon)})$, $\psi\gneq 0$, with $K=\mbox{supp}(\psi)$, by using a mollification argument on $u_\varepsilon$ as in Lemma \ref{lemma2} and integrating by parts we can easily prove that
\begin{equation*}\begin{split}
\iint_{Q_{K}}&\frac{|u_\varepsilon(x)-u_\varepsilon(y)|^{p-2}(u_\varepsilon(x)-u_\varepsilon(y))(\psi(x)-\psi(y))}{(|x-y|+\rho)^{n+sp}}\, dx\, dy\\
&\quad\geq \int_{K} \left(\int_{\R^n}\frac{|u_\varepsilon(x)-u_\varepsilon(y)|^{p-2}(u_\varepsilon(x)-u_\varepsilon(y))}{(|x-y|+\rho)^{n+sp}}dy\right)\psi(x)\, dx.
\end{split}\end{equation*}
Using the boundedness, decay and Lipschitz continuity of $u_\varepsilon$ we can make $\rho\to 0$ on the left hand side to obtain the left hand side of \eqref{lemma33}. To analyze the other integral 
we denote 
$$F_\rho(x):=\int_{\R^n}\frac{|u_\varepsilon(x)-u_\varepsilon(y)|^{p-2}(u_\varepsilon(x)-u_\varepsilon(y))}{(|x-y|+\rho)^{n+sp}}dy{,\quad x\in K:=\mbox{supp}(\psi)}.$$
We aim to prove that there exists a positive function $G\in L^{1}(K)$, {independent of $\rho$,} such that $F_\rho(x)\psi(x)\geq -G(x)$, $x\in K$, so that 
\begin{equation}\label{objetivo1}
\liminf_{\rho\to 0}\int_K F_\rho(x)\psi(x)\, dx\geq \int_K\liminf_{\rho\to 0} F_\rho(x)\psi(x)\, dx,
\end{equation}
as a consequence of a generalization of  Fatou's lemma. Let us fix $x\in K$ and $\eta>0$ such that $B_\eta(x)\subset\Omega_{r(\varepsilon)}$. First notice that
\begin{equation}\label{boundExt}
\bigg|\int_{\R^n\setminus B_{\eta}(x)}\frac{|u_\varepsilon(x)-u_\varepsilon(y)|^{p-2}(u_\varepsilon(x)-u_\varepsilon(y))}{(|x-y|+\rho)^{n+sp}}dy\bigg|\leq C,
\end{equation}
independently of $\rho$ and $x$. {By (v) of Lemma \ref{propInf},} for every $x\in \Omega_{r(\varepsilon)}$ there exists $\hat{x}\in B_{r(\epsilon)}(x)$ such that
$$u_\varepsilon(x)=\varphi_{\hat{x}}(x),\quad \mbox{ with }\quad \varphi_z(y):=u(z)+\frac{|y-z|^q}{q\varepsilon^{q-1}},\;\;y, z\in \R^n.$$
Hence by definition, for $y\in \Omega_{r(\varepsilon)}$ we have
$$u_\varepsilon(x)-u_\varepsilon(y)=\varphi_{\hat{x}}(x)-\inf_{z\in \mathbb{R}^{n}}\varphi_z(y)\geq \varphi_{\hat{x}}(x)-\varphi_{\hat{x}}(y),$$
and, by Lemma \ref{lemmaL},
\begin{equation}\label{fase0}L(u_\varepsilon(x)-u_\varepsilon(y))- L(\varphi_{\hat{x}}(x)-\varphi_{\hat{x}}(y))\geq 0,
\end{equation}
{where $L$ was given in \eqref{laL}}. Thus adding and substracting $L(\varphi_{\hat{x}}(x)-\varphi_{\hat{x}}(y))$ we get
$$\int_{B_{\eta}(x)}\frac{L(u_\varepsilon(x)-u_\varepsilon(y))}{(|x-y|+\rho)^{n+sp}}dy\geq \int_{B_{\eta}(x)}\frac{L(\varphi_{\hat{x}}(x)-\varphi_{\hat{x}}(y))}{(|x-y|+\rho)^{n+sp}}dy.$$
{It is clear that} if we prove the existence of a positive constant $C$, independent of $x$, $\hat{x}$ and $\rho$, such that
\begin{equation}\label{boundVarphix}
\bigg|\int_{B_{\eta}(x)}\frac{L(\varphi_{\hat{x}}(x)-\varphi_{\hat{x}}(y))}{(|x-y|+\rho)^{n+sp}}dy\bigg|\leq C,
\end{equation}
this, together with \eqref{boundExt}, would imply 
$$F_\rho(x)\psi(x)\geq -C\psi(x),\quad 0\lneq\psi\in C^{\infty}_{0}(\Omega_{r(\varepsilon)})\subseteq L^{1}(K),$$
with $C$ independent of $\rho$ {as wanted}. 
To show \eqref{boundVarphix} we notice that, for every $y\in \Omega_{r(\varepsilon)}$, using that $q>2$, we have
\begin{equation}\label{prop}
\|\varphi_{\hat{x}}\|_{L^\infty(\Omega)}\leq C, \quad |\nabla \varphi_{\hat{x}}(y)|=\frac{|\hat{x}-y|^{q-1}}{\varepsilon^{q-1}}\leq C,
\end{equation}
and
$$-C I\leq -\frac{q-1}{\varepsilon^{q-1}}|\hat{x}-y|^{q-2}I\leq D^2\varphi_{\hat{x}}(y)\leq \frac{q-1}{\varepsilon^{q-1}}|\hat{x}-y|^{q-2}I\leq CI,$$
where $I$ is the identity matrix and $C$ is a constant independent of $\hat{x}$ and $y$. 

Fix $\eta_0>0$ small and suppose first that $x\notin B_{\eta_0}(\hat{x})$, proceeding as in the proof of \cite[Lemma 3.6]{KKL}, since $p<2$ by using \eqref{prop} we get
\begin{equation}\label{namaste}
\begin{split}
&\bigg|\int_{B_{\eta}(x)}\frac{L(\varphi_{\hat{x}}(x)-\varphi_{\hat{x}}(y))}{(|x-y|+\rho)^{n+sp}}dy\bigg|\\
&\qquad \leq c{\int_{B_{\eta}(x)}\frac{\left(|\nabla\varphi_{\hat{x}}(x)\cdot(y-x)|+\tau(y)|y-x|^2\right)^{p-2}\tau(y)|y-x|^2}{|x-y|^{n+sp}}dy}\\
&\qquad \leq c{\tau_\infty} \int_0^\eta\left(1+\frac{{\tau_\infty} r}{|\nabla \varphi_{\hat{x}}(x)|}\right)^{p-2}|\nabla \varphi_{\hat{x}}(x)|^{p-2}r^{p(1-s)}\frac{dr}{r}\\
&\qquad \leq c{\tau_\infty} \int_0^\eta|\nabla \varphi_{\hat{x}}(x)|^{p-2}r^{p(1-s)}\frac{dr}{r}\\
&\qquad {\leq c{\tau_\infty}\left( \frac{\eta_0^{q-1}}{\varepsilon^{q-1}}\right)^{p-2}\eta^{p(1-s)}\leq C,}
\end{split}\end{equation}
where 
\begin{equation}\label{eltau}
{\tau(y):=\sup_{|z-x|<|y-x|}|D^2 \varphi_{\hat{x}}(z)|},\qquad {\tau_\infty}:=\sup_{\Omega_{r(\varepsilon)}}|D^2\varphi_{\hat{x}}|,
\end{equation} 
with $\tau_\infty$ independent of $\hat{x}$. Therefore $C$ is independent of $x$, $\hat{x}$ and $\rho$. 

Assume now $x\in B_{\eta_0}(\hat{x})$ and proceed as in \cite[Lemma 3.7]{KKL}. 
Noticing that, since $q>2,$
\begin{equation*}\begin{split}
\sup_{y\in B_r(x)}|D^2\varphi_{\hat{x}}(y)|&\leq  \sup_{y\in B_r(x)}c|y-\hat{x}|^{q-2}\leq \sup_{y\in B_r(x)}c(|y-x|+|x-\hat{x}|)^{q-2}\\
&\leq c(r+|x-\hat{x}|)^{q-2},
\end{split}\end{equation*}
where $c=c(q,\varepsilon)$,  as in \eqref{namaste} we can estimate
\begin{equation}\begin{split}\label{compAbove}
&\bigg|\int_{B_{\eta}(x)}\frac{L(\varphi_{\hat{x}}(x)-\varphi_{\hat{x}}(y))}{(|x-y|+\rho)^{n+sp}}dy\bigg|\\
&\quad \leq C{\int_{B_{\eta}(x)}\frac{\left(|\nabla\varphi_{\hat{x}}(x)\cdot(y-x)|+\tau(y)|y-x|^2\right)^{p-2}\tau(y)|y-x|^2}{|x-y|^{n+sp}}dy}\\
&\quad \leq C\int_0^\eta \left(1+\frac{(r+|x-\hat{x}|)^{q-2}r}{|\nabla\varphi_{\hat{x}}(x)|}\right)^{p-2}(r+|x-\hat{x}|)^{q-2}|\nabla\varphi_{\hat{x}}(x)|^{p-2}r^{p(1-s)}\frac{dr}{r}\\
&\quad =C\bigg(\underbrace{\int_0^{|x-\hat{x}|}dr}_{I_1}+\underbrace{\int_{|x-\hat{x}|}^\eta dr}_{I_2}\bigg),
\end{split}\end{equation}
with $C$ independent of $x$, $\hat{x}$ and $\rho$. In the first case, using the precise expression of $|\nabla\varphi_{\hat{x}}(x)|$ {given in \eqref{prop}}, we get
\begin{equation*}\begin{split}
I_1&\leq C\int_0^{|x-\hat{x}|}|x-\hat{x}|^{q-2}|\nabla\varphi_{\hat{x}}(x)|^{p-2}r^{p(1-s)}\frac{dr}{r}\\
&=C|x-\hat{x}|^{q-2+(q-1)(p-2)}\int_0^{|x-\hat{x}|}r^{p(1-s)}\frac{dr}{r}\\
&=C|x-\hat{x}|^{q(p-1)-sp}\leq C{\eta_0}^{q(p-1)-sp },
\end{split}\end{equation*}
and in the second
\begin{equation*}\begin{split}
I_2&\leq C\int_{|x-\hat{x}|}^\eta\left(\frac{r^{q-1}}{|\nabla \varphi_{\hat{x}}(x)|}\right)^{p-2}r^{q-2}|\nabla \varphi_{\hat{x}}(x)|^{p-2}\frac{r^{p(1-s)}}{r}\\
&=C\int_{|x-\hat{x}|}^\eta r^{q(p-1)-sp}\frac{dr}{r}\leq C\eta^{q(p-1)-sp},
\end{split}\end{equation*}
{with $C$ independent of $x$, $\hat{x}$ and $\rho$.} Therefore, since $q>\frac{sp}{p-1}$,  \eqref{boundVarphix} follows, and so \eqref{objetivo1}. 
We finally want to prove that 
\begin{equation*}\begin{split}
{\liminf_{\rho\to 0} F_\rho(x)}&=\liminf_{\rho\to 0} \int_{\R^n}\frac{|u_\varepsilon(x)-u_\varepsilon(y)|^{p-2}(u_\varepsilon(x)-u_\varepsilon(y))}{(|x-y|+\rho)^{n+sp}}dy\\
&\quad \geq \int_{\R^n}\frac{|u_\varepsilon(x)-u_\varepsilon(y)|^{p-2}(u_\varepsilon(x)-u_\varepsilon(y))}{|x-y|^{n+sp}}dy,\quad  x\in K.
\end{split}\end{equation*}
Notice that, by symmetry,
$$F_\rho(x)=\int_{\R^n}\frac{L(u_\varepsilon(x)-u_\varepsilon(y))-L(-\nabla \varphi_{\hat{x}}(x)\cdot(y-x))\chi_{B_\eta(x)}(y)}{(|x-y|+\rho)^{n+sp}}dy,$$
and thus we will pass to the limit by proving that there exists $0<H\in L^{1}(\mathbb{R}^{n})$, independent of $\rho$, such that
\begin{equation}\label{Fatou2}
\frac{L(u_\varepsilon(x)-u_\varepsilon(y))-L(-\nabla\varphi_{\hat{x}}(x)\cdot(y-x))\chi_{B_\eta(x)}(y)}{(|x-y|+\rho)^{n+sp}}\geq -H(y),
\end{equation}
and using once again Fatou's lemma. 
{In fact, on one hand, if} $y\in \R^n\setminus B_\eta(x)$ we easily obtain that
$$\left|\frac{{L(u_\varepsilon(x)-u_\varepsilon(y))-L(-\nabla\varphi_{\hat{x}}(x)\cdot(y-x))\chi_{B_\eta(x)}(y)}}{(|x-y|+\rho)^{n+sp}}\right|\leq C\frac{1+|u_\varepsilon(y)|^{p-1}}{|x-y|^{n+sp}}=:H_1(y),$$
which is integrable in ${\R^n\setminus B_{\eta}(x)}$. On the other hand, for $y\in B_\eta (x)$, {by \eqref{fase0}}, we have
\begin{equation*}\begin{split}
&\frac{L(u_\varepsilon(x)-u_\varepsilon(y))-L(-\nabla \varphi_{\hat{x}}(x)\cdot(y-x))}{(|x-y|+\rho)^{n+sp}}\\
&\quad \quad \geq \frac{L(\varphi_{\hat{x}}(x)-\varphi_{\hat{x}}(y))-L(-\nabla \varphi_{\hat{x}}(x)\cdot(y-x))}{(|x-y|+\rho)^{n+sp}}\\
&\quad \quad \geq -C{\frac{\left(|\nabla\varphi_{\hat{x}}(x)\cdot(y-x)|+\tau(y)|y-x|^2\right)^{p-2}\tau(y)|y-x|^2}{|x-y|^{n+sp}}}=: -H_2(y),
\end{split}\end{equation*}
where $\tau(y)$ {was given in \eqref{eltau}}. Proceeding as in \eqref{compAbove} we deduce that $H_2\in L^1(B_\eta(x))$ and hence, defining
$$H(y):=\chi_{\R^n\setminus B_\eta(x)}(y)H_1(y)+\chi_{B_\eta(x)}(y)H_2(y)\in L^1(\R^n),$$
\eqref{Fatou2} holds and {the} result follows.
\end{proof}

Before moving to the proof of Theorem \ref{ViscToWeak} we need to establish a convergence result for the right hand side of the problem.
\begin{lemma}\label{lemma3}
Let $u\in{W^{s,p}(\Omega)}\cap {L_{loc}^\infty(\R^n})$ and $f=f(x,t,\eta)$ uniformly continuous in $\Omega\times\R\times \R^n$ and Lipschitz continuous in $\eta$ satisfying \eqref{growth}. Let $\psi\in C_0^\infty(\Omega),\; \psi \gneq 0$, with $K={\rm supp}(\psi)\subseteq\Omega$. If 
\begin{equation}\label{condD}
\lim_{\varepsilon\to 0}\int_K\int_{\R^n}\frac{|u_\varepsilon(x)-u_\varepsilon(y)-(u(x)-u(y))|^p}{|x-y|^{n+sp}}dx\, dy=0,
\end{equation}
then 
$$\lim_{\varepsilon\to 0}\int_{K} f_\varepsilon(x,u_\varepsilon,D_s^pu_\varepsilon)\psi \,dx=\int_{K} f(x,u,D_s^pu)\psi  \,dx,$$
where $u_{\varepsilon}$ and $f_{\varepsilon}$ are given in \eqref{infConv} and \eqref{f_eps} respectively.
\end{lemma}

\begin{proof}
Let $\varepsilon>0$ be a fixed but arbitrarily small constant. Proceeding as in \cite[Lemma 2.2]{MO} we consider $\psi\gneq 0$ a regular function with $K:={\rm supp}(\psi)$. Using the uniform continuity property of $f$, for every $\rho>0$ there exists $\delta>0$ such that
$$|f(x,u_\varepsilon,D_s^pu_\varepsilon)-f(y,u_\varepsilon,D_s^pu_\varepsilon)|\leq \rho,\quad y\in B_{\delta}(x).$$
Since by (i) of Lemma \ref{propInf} we can guarantee that $r(\varepsilon)<\delta$ by choosing $\varepsilon$ small enough,  taking the infimum in the previous expression we get that
\begin{equation}\label{lim1}
\int_{K}|f(x,u_\varepsilon,D_s^pu_\varepsilon)-f_\varepsilon(x,u_\varepsilon,D_s^pu_\varepsilon)|\psi \, dx\leq \rho\|\psi \|_{L^\infty(K)}|K|,
\end{equation}
with $\rho$ arbitrarily small. Let us see that 
\begin{equation}\label{lim2}
\lim_{\varepsilon\to 0}\int_K f(x,u_\varepsilon,D_s^pu_\varepsilon)\psi\, dx=\int_K f(x,u,D_s^p u)\psi .
\end{equation}
On one hand, since $\|u_\varepsilon\|_{L^\infty(K)}\leq C$ uniformly {(see Remark \ref{InfBound})}, it follows that
\begin{equation}\label{lagamma}
\max_{t\in(-\|u_\varepsilon\|_{L^\infty(\Omega)},\|u_\varepsilon\|_{L^\infty(\Omega)})}|\gamma(t)|\leq \max_{t\in(-\|u\|_{L^\infty(\Omega)},\|u\|_{L^\infty(\Omega)})}|\gamma(t)|.
\end{equation}
Then, due to the growth of $f$, we have that
$$|f(x,u_\varepsilon(x),D_s^p u(x))|\leq C|D_s^p u(x)|^{\frac{p-1}{p}}+\phi(x)\in L^1(K),$$
with $C$ independent of $\varepsilon$, and thus
\begin{equation}\label{D1}
\lim_{\varepsilon\to 0}\int_K f(x,u_\varepsilon,D_s^pu)\psi  dx=\int_K f(x,u,D_s^p u)\psi .
\end{equation}
Besides, since $f$ is Lipschitz continuous in the third variable,
\begin{equation}\begin{split}
&\int_K|f(x,u_\varepsilon,D_s^pu_\varepsilon)- f(x,u_\varepsilon,D_s^p u)|\psi \\
&\quad \leq C\int_K|D_s^pu_\varepsilon-D_s^pu| dx\\
&\quad {\leq C\int_K\int_{\R^n}\frac{||u_\varepsilon(x)-u_\varepsilon(y)|^{p}-|u(x)-u(y)|^{p}|}{|x-y|^{n+sp}}dy\, dx}\\
&\quad \leq C\left(\int_K\int_{\R^n}\frac{|g_{\varepsilon}(x,y)|^p}{|x-y|^{n+sp}}dy\, dx\right)^{1/p}\left(\int_K\int_{\R^n}\frac{|u(x)-u(y)|^{p}+|g_{\varepsilon}(x,y)|^{p}}{|x-y|^{n+sp}}dy\, dx\right)^{\frac{p-1}{p}},\label{D2}
\end{split}\end{equation}
where 
$$g_{\varepsilon}(x,y):=u_\varepsilon(x)-u_\varepsilon(y)-(u(x)-u(y)),\, x\in K,\, y\in\R^{N},$$
and in the last step we have applied \eqref{showman}
and the H\"older's inequality. Hence, since {$u\in W^{s,p}(\Omega)$}, putting together \eqref{D1}, \eqref{D2}, and using the hypothesis \eqref{condD}, we conclude
\begin{equation*}\begin{split}
&\lim_{\varepsilon\to 0}\int_K|f(x,u_\varepsilon,D_s^pu_\varepsilon)-f(x,u,D_s^p u_\varepsilon)|\psi  dx\\
&\quad \leq \lim_{\varepsilon\to 0}\int_K|f(x,u_\varepsilon,D_s^pu_\varepsilon)-f(x,u_\varepsilon,D_s^p u)|\psi  dx\\
&\quad\quad +\lim_{\varepsilon\to 0}\int_K|f(x,u_\varepsilon,D_s^pu)- f(x,u,D_s^p u)|\psi dx\\
&\quad =0,
\end{split}\end{equation*}
{proving \eqref{lim2}. The result follows as a consequence of \eqref{lim1} and \eqref{lim2}.}
\end{proof}

\subsection{Proof of Theorem \ref{ViscToWeak}}
Let $u_\varepsilon$ given by \eqref{infConv}. By Lemma \ref{lemma1} it is a viscosity supersolution of
$$(-\Delta)^s_pu_\varepsilon=f_\varepsilon(x,u_\varepsilon,D_s^pu_\varepsilon) \quad \mbox{ in }\Omega_{r(\varepsilon)},$$
and
$$(-\Delta)_s^p u_\varepsilon(x)\geq f_\varepsilon(x,u_\varepsilon(x),D_s^pu_\varepsilon(x)) \quad\mbox{ a.e. in }\Omega_{r(\varepsilon)}.$$
Thus, by Lemma \ref{lemma2} and Lemma \ref{lemma22}, for all $1<p<\infty$ and every $\psi\in C_0^\infty(\Omega_{r(\varepsilon)})$, $\psi\gneq 0$, we have
\begin{equation}\label{ineqEps}
\begin{split}
\iint_{Q_{\Omega_{r(\varepsilon)}}}&\frac{|u_\varepsilon(x)-u_\varepsilon(y)|^{p-2}(u_\varepsilon(x)-u_\varepsilon(y))(\psi(x)-\psi(y))}{|x-y|^{n+sp}}dx\, dy\\
&\geq \int_{\Omega_{r(\varepsilon)}} f_\varepsilon(x,u_\varepsilon,D_s^pu_\varepsilon)\psi dx,
\end{split}
\end{equation}
that is, $u_{\varepsilon}$ is a weak supersolution in $\Omega_{r(\varepsilon)}$. Let $\varPhi\in C_0^\infty(\Omega)$, $\varPhi\geq 0$ be a fixed but arbitrarily smooth function with $K:={\rm supp\, }\varPhi$. By taking $\varepsilon$ small enough we can assume, without loss of generality, that $K\subset\Omega_{r(\varepsilon)}$. Our objective is passing to the limit in $\varepsilon$ in {\eqref{ineqEps}} to prove that
\begin{equation}\label{objetivo}
\iint_{Q_K}\frac{|u(x)-u(y)|^{p-2}(u(x)-u(y))(\varPhi(x)-\varPhi(y))}{|x-y|^{n+sp}}dx\, dy\geq \int_K f(x,u,D_s^pu)\varPhi\, dx,
\end{equation}
 by using Lemma \ref{lemma3} and some integral estimates. Let us first consider a non negative smooth function $\xi$, $0\leq\xi\leq 1$, with support in a compact domain $K''\subset \Omega_{r(\varepsilon)}$ and such that $\xi=1$ in a compact $K'\subset K''$ with $K\subset K'\subset K''$. Thus, by Proposition \ref{Caccioppoli},
\begin{equation*}\begin{split}
&\int_{K'}\int_{\R^n}\frac{|u_\varepsilon(x)-u_\varepsilon(y)|^p}{|x-y|^{n+sp}}dy\, dx=\int_{K'}\int_{\R^n}\frac{|u_\varepsilon(x)-u_\varepsilon(y)|^p}{|x-y|^{n+sp}}\xi^p(x)dy\, dx\\
&\leq C\left[(\mbox{osc}(u_\varepsilon))^p\left(\int_{K''}\int_{\R^n}\frac{|\xi(x)-\xi(y)|^p}{|x-y|^{n+sp}}dy\, dx+{\gamma_{\infty,u_\varepsilon}^p}\right)+\mbox{osc}(u_\varepsilon)\right],
\end{split}\end{equation*}
with $C$ independent of $\varepsilon$ and $\gamma_{\infty,u_\varepsilon}:=\max_{t\in(-\|u_\varepsilon\|_{L^\infty(\Omega)},\|u_\varepsilon\|_{L^\infty(\Omega)})}|\gamma(t)|$. Notice that, since $u_\varepsilon$ is increasing {as $\varepsilon\to 0$},
$$\mbox{osc}(u_\varepsilon)\leq \sup_{\R^n} u-\inf_{\R^n} u_{\varepsilon_0},\quad \varepsilon<\varepsilon_0,$$
and $\|u_\varepsilon\|_{L^\infty(\Omega)}\leq \|u\|_{L^\infty(\Omega)}$. Thus by \eqref{lagamma} we get that
\begin{equation}\label{covid}
\int_{K'}\int_{\R^n}\frac{|u_\varepsilon(x)-u_\varepsilon(y)|^p}{|x-y|^{n+sp}}dy\, dx\leq C,
\end{equation}
with $C$ independent of $\varepsilon$ or, equivalently,
\begin{equation}\label{Lpunif}
\left\|\frac{u_\varepsilon(x)-u_\varepsilon(y)}{|x-y|^{\frac{n+sp}{p}}}\right\|_{L^p(K'\times \R^n)}\leq C.
\end{equation}
Since we already knew that $u_\varepsilon\to u$ a.e. as $\varepsilon\to 0$ we deduce that
\begin{equation}\label{Lpconv}
u_\varepsilon\to u \mbox{ in } L^q(K'), \mbox{ for all } q\in [1, p^*_s),
\end{equation}
$$\frac{u_\varepsilon(x)-u_\varepsilon(y)}{|x-y|^{\frac{n+sp}{p}}}\rightharpoonup \frac{u(x)-u(y)}{|x-y|^{\frac{n+sp}{p}}}\mbox{ in }L^p(K'\times \R^n),$$
where $p_s^*:=\displaystyle\frac{np}{n-ps}$, $n>ps$, is the critical Sobolev exponent. Then
\begin{equation}\label{weakConv}
\int_{K'}\int_{\R^n}\frac{u_\varepsilon(x)-u_\varepsilon(y)}{|x-y|^{\frac{n+sp}{p}}}\varphi(x,y) dy\, dx\rightarrow \int_{K'}\int_{\R^n}\frac{u(x)-u(y)}{|x-y|^{\frac{n+sp}{p}}}\varphi(x,y) dy\, dx,
\end{equation}
for every $\varphi\in L^{p/(p-1)}(K'\times\R^n)$. We want to point out that, since \eqref{covid} also implies 
$$\left\|\frac{|u_\varepsilon(x)-u_\varepsilon(y)|^{p-2}(u_\varepsilon(x)-u_\varepsilon(y))}{|x-y|^{\frac{(n+sp)(p-1)}{p}}}\right\|_{L^\frac{p}{p-1}(K'\times \R^n)}\leq C,$$
we also have
\begin{equation}\label{weakConv2}
\frac{|u_\varepsilon(x)-u_\varepsilon(y)|^{p-2}(u_\varepsilon(x)-u_\varepsilon(y))}{|x-y|^{(n+sp)\frac{p-1}{p}}}\rightharpoonup \frac{|u(x)-u(y)|^{p-2}(u(x)-u(y))}{|x-y|^{(n+sp)\frac{p-1}{p}}},
\end{equation}
in $L^{p/(p-1)}(K'\times \R^n)$. Consider now the function
$$\psi(x):=(u(x)-u_\varepsilon(x))\theta(x),$$
with $\theta$ a non negative smooth function with compact support in $K'\subset\Omega$, such that $\theta\equiv 1$ in $K\subset K'$. 
Let us denote
\begin{equation*}\begin{split}
U(x,y)&:=|u(x)-u(y)|^{p-2}(u(x)-u(y)),\\
U_\varepsilon(x,y)&:= |u_\varepsilon(x)-u_\varepsilon(y)|^{p-2}(u_\varepsilon(x)-u_\varepsilon(y)),
\end{split}\end{equation*}
and notice that
$$\psi(x)-\psi(y)=\theta(x)(u(x)-u_\varepsilon(x)-(u(y)-u_\varepsilon(y))+(\theta(x)-\theta(y))(u(y)-u_\varepsilon(y)).$$
Thus, replacing in \eqref{ineqEps} we can write
\begin{equation*}\begin{split}
\int_{K'}f_\varepsilon(x,u_\varepsilon&,D_s^pu_\varepsilon)\psi\, dx\leq \iint_{Q_{K'}}\frac{U_\varepsilon(x,y)(\psi(x)-\psi(y))}{|x-y|^{n+sp}}\, dy\, dx\\
=&\underbrace{\iint_{Q_{K'}}\frac{U_\varepsilon(x,y)(\theta(x)-\theta(y))(u(y)-u_\varepsilon(y))}{|x-y|^{n+sp}}\, dy\, dx}_{I_{1}}\\
&+\underbrace{\int_{K'}\int_{\R^n}\frac{U(x,y)\theta(x)(u(x)-u(y)-(u_\varepsilon(x)-u_\varepsilon(y))}{|x-y|^{n+sp}}\, dy\, dx}_{I_{2}}\\
&-\underbrace{\int_{K'}\int_{\R^n}\frac{(U(x,y)-U_\varepsilon(x,y))\theta(x)(u(x)-u(y)-(u_\varepsilon(x)-u_\varepsilon(y))}{|x-y|^{n+sp}}\, dy\, dx}_{I_{3}}.
\end{split}\end{equation*}
By H\"older inequality and the uniform bound \eqref{Lpunif} it follows that
\begin{equation*}\begin{split}
I_{1}&\leq \left(\iint_{Q_{K'}}\frac{|u_\varepsilon(x)-u_\varepsilon(y)|^p}{|x-y|^{n+sp}}\, dy\, dx\right)^{\frac{p-1}{p}}\left(\iint_{Q_{K'}}\frac{|\theta(x)-\theta(y)|^p}{|x-y|^{n+sp}}|u(y)-u_\varepsilon(y)|^p\, dy\, dx\right)^{\frac{1}{p}}\\
&\leq C\left(\iint_{Q_{K'}}\frac{|\theta(x)-\theta(y)|^p}{|x-y|^{n+sp}}|u(y)-u_\varepsilon(y)|^p\, dy\, dx\right)^{\frac{1}{p}}.
\end{split}\end{equation*}
Since by the boundedness and smoothness of $\theta$ and Remark \ref{InfBound} we have
$$\frac{|\theta(x)-\theta(y)|^p}{|x-y|^{n+sp}}|u(y)-u_\varepsilon(y)|^p
\leq C(p)\|u\|_{L^{\infty}(\R^n)}\frac{|\theta(x)-\theta(y)|^p}{|x-y|^{n+sp}}\in L^{1}(Q_{K'}),$$
by the Dominated Convergence Theorem we get that $I_{1}\to 0$ as $\varepsilon\to 0$. 
Noticing that 
$$\frac{U(x,y)\theta(x)}{|x-y|^{(n+sp)\frac{p-1}{p}}}=\frac{|u(x)-u(y)|^{p-2}(u(x)-u(y))\theta(x)}{|x-y|^{(n+sp)\frac{p-1}{p}}}\in L^{\frac{p}{p-1}}(K'\times\R^n),$$
then $I_{2}\to 0$ as $\varepsilon\to 0$ as a consequence of \eqref{weakConv}. 

Let us analyze $I_3$. By performing some algebraic computations it can be seen that
\begin{equation*}\begin{split}
&(U(x,y)-U_\varepsilon(x,y))(u(x)-u(y)-(u_\varepsilon(x)-u_\varepsilon(y))\\
&\quad\geq(|u(x)-u(y)|-|u_\varepsilon(x)-u_\varepsilon(y)|)(|u(x)-u(y)|^{p-1}-|u_\varepsilon(x)-u_\varepsilon(y)|^{p-1})\geq 0,
\end{split}\end{equation*}
and hence we obtain
$$0\leq \lim_{\varepsilon\to 0}I_3\leq \limsup_{\varepsilon\to 0}\left(-\int_{K'} f_\varepsilon(x,u_\varepsilon,D_s^pu_\varepsilon)\psi dx\right).$$
By \eqref{growth} and \eqref{lagamma} we obtain
\begin{equation}\label{estf}
-\int_{K'} f_\varepsilon(x,u_\varepsilon,D_s^pu_\varepsilon)\psi\leq \gamma_\infty \int_{K'} |D_s^pu_\varepsilon|^{\frac{p-1}{p}}(u-u_\varepsilon)\theta dx+\|\phi\|_{L^\infty(K')}\int_{K'}(u-u_\varepsilon)\theta dx,
\end{equation}
where $\gamma_\infty:=\sup_{x\in K'} {|\gamma (u(x))|}$. Since
\begin{equation*}\begin{split}
\int_{K'} |D_s^pu_\varepsilon|^{\frac{p-1}{p}}(u-u_\varepsilon)\theta dx &\leq \left(\int_{K'}|D_s^pu_\varepsilon| dx\right)^{\frac{p-1}{p}}\left(\int_{K'} |u-u_\varepsilon|^p\right)^{\frac{1}{p}}\\
&\leq \left(\int_{K'}\int_{\R^n}\frac{|u_\varepsilon(x)-u_\varepsilon(y)|^p}{|x-y|^{n+sp}}\, dy\, dx\right)^{\frac{p-1}{p}}\left(\int_{K'} |u-u_\varepsilon|^p\right)^{\frac{1}{p}},
\end{split}\end{equation*}
the right hand side of \eqref{estf} goes to 0 as $\varepsilon\to 0$ as a consequence of \eqref{Lpunif} and \eqref{Lpconv}. Thus
\begin{equation}\label{I3}
\lim_{\varepsilon\to 0}I_3=0.
\end{equation}
This allows us to have the hypothesis of Lemma \ref{lemma3} satisfied for every $1<p<\infty$. Indeed if $p\geq 2$, applying the algebraic inequality
$$2^{p-1}|a-b|^p\leq ||a|^{p-2}a-|b|^{p-2}b|(a-b),\, a,\, b\in\R,$$
we have 
$$0\leq \int_{K}\int_{\R^n}\frac{|u(x)-u(y)-(u_\varepsilon(x)-u_\varepsilon(y))|^p}{|x-y|^{n+sp}}\, dy\, dx\leq I_3,$$
and hence by \eqref{I3}
\begin{equation}\label{strong}
\lim_{\varepsilon\to 0}\int_{K}\int_{\R^n}\frac{|u(x)-u(y)-(u_\varepsilon(x)-u_\varepsilon(y))|^p}{|x-y|^{n+sp}}, dy\, dx=0.
\end{equation}
If $1<p<2$ we use the vector inequality 
$$\frac{|a-b|^2}{(|a|+|b|)^{2-p}}\leq C(|a|^{p-2}a-|b|^{p-2}b)\cdot (a-b),\quad C=C(n,p), \;a,b\in\R^n,$$
to get
\begin{equation*}\begin{split}
\int_{K}\int_{\R^n}&\frac{|u(x)-u(y)-(u_\varepsilon(x)-u_\varepsilon(y))|^p}{|x-y|^{n+sp}}\, dy\, dx\\
&\leq \left(\int_{K}\int_{\R^n}\frac{|u(x)-u(y)-(u_\varepsilon(x)-u_\varepsilon(y))|^2}{|x-y|^{n+sp}(|u(x)-u(y)|+|u_\varepsilon(x)-u_\varepsilon(y)|)^{2-p}}\, dy\, dx\right)^{p/2}\\
&\quad \cdot \left(\int_{K}\int_{\R^n}\frac{(|u(x)-u(y)|+|u_\varepsilon(x)-u_\varepsilon(y)|)^{p}}{|x-y|^{n+sp}}\, dy\, dx\right)^{\frac{2-p}{2}}\\
&\leq C\left(\int_{K}\int_{\R^n}\frac{|u(x)-u(y)-(u_\varepsilon(x)-u_\varepsilon(y))|^2}{|x-y|^{n+sp}(|u(x)-u(y)|+|u_\varepsilon(x)-u_\varepsilon(y)|)^{2-p}}\, dy\, dx\right)^{p/2}\\
&\leq C\left(\int_{K}\int_{\R^n}\frac{(U(x,y)-U_\varepsilon(x,y))(u(x)-u(y)-(u_\varepsilon(x)-u_\varepsilon(y))}{|x-y|^{n+sp}}\, dy\, dx\right),
\end{split}\end{equation*}
where in the first inequality we have applied H\"older's inequality with $(2/p, 2/(2-p))$ and \eqref{covid} in the second one. Then \eqref{strong} also holds for $1<p<2$ as a consequence of \eqref{I3}.

Thus we can apply Lemma \ref{lemma3} to pass to the limit in the right hand side of \eqref{ineqEps} and, by using \eqref{weakConv2} with
$$\varphi(x,y):=\frac{\varPhi(x)-\varPhi(y)}{|x-y|^{\frac{n+sp}{p}}}\in L^{p}(K\times\R^n),$$
we conclude \eqref{objetivo},
what completes the proof of Theorem \ref{ViscToWeak}.

\section{Weak solutions are viscosity solutions}
As mentioned in the Introduction, to prove Theorem \ref{WeakToVisc} we need to use a strategy different from \cite[Theorem 1.5]{MO}. The nonlocal behavior of the problem (in particular the fact that $D_s^p$ is given by an integral defined in the whole $\R^n$), creates difficulties that cannot be solved following the local approach. In particular, the continuity of the operator $D_s^p$ under regular perturbations is needed.

\subsection{Proof of Theorem \ref{WeakToVisc}} 

Let $u$ be a weak supersolution of \eqref{problem}, continuous in $\Omega$. We proceed by contradiction, that is, we suppose that there exist a point $x_0\in \Omega$ and a function $\phi$, with $\phi(x_0)=u(x_0)$, satisfying the hypotheses in (iii) of Definition \ref{defVisc} such that
\begin{equation}\label{hipCont}
\psi_r(x):=\begin{cases}
\phi(x)\leq u(x),\quad x\in B_{r}(x_0),\\
u(x),\quad x\in \mathbb{R}^{n}\setminus B_{r}(x_0),\\
\end{cases}
\end{equation}
and
\begin{equation}\label{contr}
(-\Delta)^s_p\psi_r(x_0)<f(x_0,\psi_r(x_0),D_s^p\psi_r(x_0))=f(x_0,u(x_0),D_s^p\psi_r(x_0)).
\end{equation}
Since $\psi_r\in C^2(B_r(x_0)){\cap L^{p-1}_{sp}(\R^n)}$, the operator $D_s^p\psi_r$ is continuous in $B_r(x_0)$ and thus, by the hypotheses on $f$ and $u$, the map
$$x\mapsto f(x,u(x),D_s^p\psi_r(x)),$$
is also continuous in $B_r(x_0)$. Thus, by the continuity of the operator $(-\Delta)^{s}_{p}$ and \eqref{contr}, there exist $\mu_0>0$ and $0<r_1<r$ such that
\begin{equation}\label{contrBall}
(-\Delta)^s_p\psi_r(x)\leq f(x,u(x),D_s^p\psi_r(x))-\mu_0,\qquad x\in B_{r_1}(x_0).
\end{equation}
Moreover, since $\psi_r\in L^{p-1}_{sp}(\mathbb{R}^{n})$, by \cite[Lemma 3.9]{KKL} we know that, for every $\varepsilon>0$ and $\rho>0$, there exist $0<\theta_1=\theta_1(\varepsilon,\rho)$, $0<r_2<\rho$ and $\eta\in C^{2}_{0}(B_{\frac{r_2}{2}}(x_0))$, $0\leq \eta\leq 1$ with $\eta(x_0)=1$ such that,
$$\sup_{B_{r_2}(x_0)}|(-\Delta)^s_{p}\psi_r-(-\Delta)^{s}_{p}(\psi_r+\theta\eta)|<\varepsilon,$$
for every $0<\theta<\theta_1$. 
Thus, taking $\varepsilon=\mu_0/2$, and $\rho=r_1$ using the Lipschitz hypothesis of $f$, by Lemma \ref{Lema1} and \eqref{contrBall}, it follows that
\begin{equation}\label{absurd}
\begin{split}
(-\Delta)^s_p(\psi_r+\theta\eta)(x)&<(-\Delta)^s_p\psi_r(x)+\frac{\mu_0}{2}\\
&\leq f(x,u(x),D_s^p\psi_r(x))-\frac{\mu_0}{2}\\
&\leq f(x,u(x),D_s^p(\psi_r+\theta\eta)(x))
\end{split}
\end{equation}
for $x\in B_{\frac{r_2}{2}}(x_0)$, $0<\theta<\min\{\theta_1(\mu_0/2, r_1),\widetilde{\theta}(\mu_0/(2K_f),r_2\}$ where $\widetilde{\theta}$ is given in Lemma \ref{Lema1} and $K_f$ is the Lipschitz constant of $f$. Thus, defining 
$$\tilde{f}(x,g):=f(x,u(x),g),$$
by testing with a smooth function compactly supported in $B_{\frac{r_2}{2}}(x_0)$ and integrating by parts, we get that \eqref{absurd} also holds in the weak sense. Thus, ${\psi}_r+\theta\eta$ is a weak subsolution of the problem
\begin{equation}\label{probAux}
(-\Delta)^s_p v(x)=\tilde{f}(x,D_s^p v(x)),\quad x\in B_{\frac{r_2}{2}}(x_0),
\end{equation}
Since, trivially, $u$ is a weak supersolution of \eqref{probAux} and 
$$ ({\psi}_r+\theta\eta)(x)={\psi}_r(x)\leq u(x),\quad x\in\mathbb{R}^{n}\setminus B_{\frac{r_2}{2}}(x_0),$$
by the (CPP) we have that $u\geq {\psi}_r+\theta\eta$ also in $B_{\frac{r_2}{2}}(x_0)$. This in particular implies that
$$u(x_0)\geq {\psi}_r(x_0)+\theta\eta(x_0)=\psi_r(x_0){+\theta>\psi_r(x_0)},$$
a contradiction with \eqref{hipCont}, what finishes the proof.

\subsection{Comparison principles}\label{sub_comp}
 A fundamental issue in Theorem \ref{WeakToVisc} is the availability of a comparison principle for weak solutions. Unfortunately not much is known yet in this line. In this subsection we provide a brief review of the {\it state-of-art} of the comparison results available for weak solutions of problems involving the fractional $p$-Laplacian.
 
The first one can be found in \cite[Lemma 9]{LL} and allows us to compare weak solutions when the right hand side of the problem is linear.
 \begin{theorem}
 Let $u$ and $v$ be two continuous functions belonging to $W_0^{s,p}(\R^n)$ and let $D\subset \R^n$ be a domain. If $v\geq u$ in $\R^n\setminus D$ and
 \begin{equation*}\begin{split}
 &\iint_{\R^{2n}}\frac{|v(x)-v(y)|^{p-2}(v(x)-v(y))(\phi(x)-\phi(y))}{|x-y|^{n+sp}}dx\, dy\\
 &\quad \geq \iint_{\R^{2n}}\frac{|u(x)-u(y)|^{p-2}(u(x)-u(y))(\phi(x)-\phi(y))}{|x-y|^{n+sp}}dx\, dy,
 \end{split}\end{equation*}
for all $\phi\in C_0(D)$, $\phi\gneq 0$, then $v\geq u$ in $D$.
 \end{theorem}
See also \cite[Proposition 2.10]{IMS} and \cite[Lemma 6]{KKP2} for similar results with weaker assumptions on $u$ and $v$. In the case of a nonlinear right hand side the following weak comparison principle is proved in \cite[Proposition 2.5]{dPQ} (see \cite[Theorem 1.1]{J} for the strong counterpart).
\begin{theorem}
Let $\Omega$ be a bounded domain and $c\in L^1_{loc}(\Omega)$. Suppose $u,v\in {W}^{s,p}(\Omega)$ are non negative super and subsolutions respectively of the problem
\begin{equation}\label{nonLinear}
(-\Delta)^s_pu=c(x)|u|^{p-2}u\quad \mbox{ in }\Omega.
\end{equation}
If $c(x)\leq 0$ in $\Omega$ and $u\geq v$ in $\R^n\setminus \Omega$ then $u\geq v$ a.e. in $\Omega$.
\end{theorem}
Notice that, under further regularity assumptions on $c$ and the solution $u$, both Theorem \ref{ViscToWeak} and Theorem \ref{WeakToVisc} apply, establishing the equivalence of weak and viscosity solutions for the problem \eqref{nonLinear}. See also \cite[Theorem 7]{B} for a comparison principle for minimizers.

As far as we know, no comparison results are available yet for nonlinearities involving any kind of gradient term. In the classical case, when $s=1$, some can be found in \cite{TPR, PS}.

\section*{Conflict of interest}
On behalf of all authors, the corresponding author states that there is no conflict of interest.


\begin{thebibliography}{1}
\setlinespacing{0.98}
\frenchspacing

\bibitem{adams} R.A. Adams, \textit{Sobolev Spaces, Pure and Applied Mathematics}, vol.65, Academic Press [A subsidiary of Harcourt Brace Jovanovich, Publishers], New York, London, 1975.



\bibitem{AF} B. Abdellaoui, A. J. Fern\'andez, \textit{Nonlinear fractional Laplacian problems with nonlocal ``gradient terms"}, to appear in Proc. Royal Soc. Edinburgh Sect. A.

\bibitem{BDS} A. Banerjee, G. D\'{a}vila, Y. Sire, \textit{Regularity for parabolic systems with critical growth in the gradient and applications}, arXiv:2005.04004.

\bibitem{BPSV} B. Barrios, I. Peral, F. Soria, E. Valdinoci, \textit{A Widder's type theorem for the heat equation with nonlocal diffusion}, Arch. Ration. Mech. Anal. 213 (2014), no. 2, 629--650.

\bibitem{BPV} B. Barrios, I. Peral, S. Vita, \textit{Some remarks about the summability of nonlocal nonlinear problems}, Adv. Nonlinear Anal. 4 (2015), no. 2, 91--107.

\bibitem{BCF} Bjorland, C., Caffarelli, L., Figalli, A. Non-local gradient dependent operators. Adv. Math. 230 (2012), no. 4-6, 1859--1894.

\bibitem{BL} L. Brasco, E. Lindgren, \textit{Higher Sobolev regularity for the fractional p-Laplace equation in the superquadratic case}, Advances in Mathematics, Vol. 304, 2 (2017), 300--354.

\bibitem{B} C. Bucur, \textit{A symmetry result in R2 for global minimizers of a general type of nonlocal energy}, Calc. Var. Partial Differential Equations 59 (2020), no. 2, Paper No. 52.

\bibitem{BS} C. Bucur, M. Squassina, \textit{An asymptotic expansion for the fractional -Laplacian and gradient dependent nonlocal operators}, arXiv:2001.09892.

\bibitem{CD} L. Caffarelli, G. D\'avila, \textit{Interior regularity for fractional systems}, Annales de l'Institut Henri Poincar\'e C, Analyse non lin\'eaire, Vol. 36(1) (2019), 165--180.

\bibitem{CS} G. Comi, G. Stefani, \textit{A distributional approach to fractional Sobolev spaces and fractional variation: Existence of blow-up}, Journal of Functional Analysis 277 (2019), 3373--3435.

\bibitem{dPQ} L. M. Del Pezzo, A. Quaas, \textit{A Hopf's lemma and a strong minimum principle for the fractional $p$-Laplacian}, J. Differential Equations, 263.1 (2017), 765--778.

\bibitem{dCLP} A. Di Castro, T. Kuusi, G. Palatucci, \textit{Nonlocal Harnack inequalities}, J. Funct. Anal. \textbf{267} (2014), no. 6, 1807--1836.

\bibitem{dCLP2}  A. Di Castro, T. Kuusi, G. Palatucci, \textit{Local behavior of fractional p-minimizers}, Ann. Inst. H. Poincar\'e Anal. Non Lin\'eaire \textbf{33} (2016), 1279--1299.

\bibitem{guida} { E. Di Nezza, G. Palatucci, E. Valdinoci},
{\em Hitchhiker's guide to the fractional Sobolev spaces.}
Bull. Sci. math., {\bf 136} (2012), no. 5, 521--573.



\bibitem{IMS} A. Iannizzotto, S. Mosconi and M. Squassina, \textit{Global Holder regularity for the fractional $p$-Laplacian}, Rev. Mat. Iberoam 32.4 (2016), 1353--1392.

\bibitem{IN} H. Ishii, G. Nakamura, \textit{A class of integral equations and approximation of p-Laplace
equations}, Calc. Var. Partial Differential Equations \text{bf} 37 (2010), no. 3-4, 485--522.

\bibitem{J} S. Jarohs, \textit{Strong comparison principle for the fractional $p$-Laplacian and applications to starshaped rings}, Advanced Nonlinear Studies 18.4 (2018), 691--704.

\bibitem{JJ} V. Julin, P. Juutinen, \textit{A new proof for the equivalence of weak and viscosity solutions for the p-Laplace equation}. Communications in PDE \textbf{37} 5 (2012), 934--946.

\bibitem{JL} P. Juutinen, P. Lindqvist, \textit{ A theorem of Rad\'o's type for the solutions of a quasi-linear equation}. Math. Res. Lett. \textbf{11}, (2004), 31-34.

\bibitem{JLM} P. Juutinen, P. Lindqvist, and J. Manfredi, \textit{On the equivalence of viscosity solutions and weak solutions for a quasilinear equation}, SIAM J. Math. Anal. \emph{33} 3 (2001), 699--717.


\bibitem{KKL} J. Korvenp\"{a}\"{a}, T. Kuusi, E. Lindgren, \textit{Equivalence of solutions to fractional $p$-Laplace type equations},  J. Math. Pures Appl. (9), 132 (2019).

\bibitem{KKP} J. Korvenp\"{a}\"{a}, T. Kuusi, G. Palatucci, \textit{The obstacle problem for nonlinear integro-differential operators}, Calc. Var. Partial Differential Equations 55 (2016), No. 3, Art. 63.

\bibitem{KKP2} J. Korvenp\"{a}\"{a}, T. Kuusi, G. Palatucci, \textit{Fractional superharmonic functions and the Perron method for nonlinear integro-differential equations}, Math. Ann. 369 (2017),
no. 3-4, 1443--1489.

\bibitem{KMS}  T. Kuusi, G. Mingione, Y. Sire, \textit{Nonlocal equations with measure data}, Comm. Math. Phys., 337 (2015), 1317--1368.



\bibitem{KMS2} T. Kuusi, G. Mingione, Y. Sire, \textit{Nonlocal self-improving properties.} Anal. PDE 8 (2015), no. 1, 57-114.

\bibitem{TPR} T. Leonori, A. Porretta and G. Riey, \textit{Comparison principles for p-Laplace equations with lower order terms}, Annali di Matematica (2016), 1--27.

\bibitem{LL} E. Lindgren and P. Lindqvist, \textit{Fractional eigenvalues}, Calc. Var. PDE 49 (2014), 795--826.

\bibitem{MO} M. Medina, P. Ochoa, \textit{On viscosity and weak solutions for non-homogeneous p-Laplace equations}, Adv. Nonlinear Anal. \textbf{8} (2019), 468--481.

\bibitem{P} G. Palatucci, \textit{The Dirichlet problem for the p-fractional Laplace equation}, Nonlinear Analysis 177 (2018), 699--732.

\bibitem{Po} A. V. Pokrovskii, \textit{Removable singularities of solutions of elliptic equations}, J.Math. Sci. \textbf{160} (2008), 61-83.

\bibitem{PS} P. Pucci and J. Serrin, \textit{The  maximum principle}, Progress in non-linear differential equations and their applications Vol. 73 Birkh\"{a}user, Boston, 2007.

\bibitem{SV} R. Servadei, E. Valdinoci, \textit{Weak and viscosity solutions of the fractional Laplace equation}, Publ. Mat. \textbf{58} (2014), 133--154.

\bibitem{SSS} A. Schikorra, T. Shieh, D. Spector, \textit{Regularity for a fractional p-Laplace equation}
Commun. Contemp. Math., 20 (01), 2018.



\end{thebibliography}
\end{document}